\documentclass[reqno,a4paper]{amsart}
\usepackage[utf8]{inputenc}
\usepackage[T1]{fontenc}
\usepackage{amssymb,amsthm,amsmath,enumerate}
\usepackage{calrsfs}
\usepackage{hyperref}
\usepackage{mathbbol}
\usepackage{mathabx}
\usepackage{mathtools}
\usepackage[all]{xy}
\usepackage{fancyvrb}
\usepackage{fvextra}
\usepackage{seqsplit}
\usepackage{stmaryrd}

\usepackage{tikz}
\usetikzlibrary{cd, arrows,arrows.meta}

\theoremstyle{plain}
\newtheorem{corollary}[subsection]{Corollary}	
\newtheorem{lemma}[subsection]{Lemma}
\newtheorem{proposition}[subsection]{Proposition}
\newtheorem{theorem}[subsection]{Theorem}

\theoremstyle{definition}
\newtheorem{definition}[subsection]{Definition}
\newtheorem{convention}[subsection]{Convention}

\theoremstyle{remark}
\newtheorem{example}[subsection]{Example}
\newtheorem{examples}[subsection]{Examples}
\newtheorem{remark}[subsection]{Remark}
\newtheorem{remarks}[subsection]{Remarks}

\newenvironment{tfae}
{
	\begin{enumerate}}
	{\end{enumerate}}

\numberwithin{equation}{section}

\newcommand{\noproof}{\hfill \qed}
\newcommand{\defn}{\textbf}

\newcommand{\To}{\Rightarrow}
\newcommand{\psj}{\curlyvee}
\newcommand{\da}{{\downarrow}}
\newcommand{\ua}{{\uparrow}}

\newcommand{\oleq}{\subseteq} %subobject
\newcommand{\normal}{\trianglelefteq} %normal subobject
\newcommand{\omeet}{\cap} %meet of subobjects 
\newcommand{\bigomeet}{\bigcap} %meet of subobjects 
\newcommand{\bigojoin}{\bigcup} %join of subobjects 
\newcommand{\ojoin}{\cup}  %join of subobjects
\newcommand{\iimplies}{\shortrightarrow} %internal implies

\newcommand{\C}{\ensuremath{\mathcal{C}}}
\newcommand{\D}{\ensuremath{\mathcal{D}}}

\newcommand{\K}{\ensuremath{\mathbb{K}}}

\newcommand{\Ab}{\ensuremath{\mathsf{Ab}}}
\newcommand{\Grp}{\ensuremath{\mathsf{Grp}}}
\newcommand{\Pt}{\ensuremath{\mathsf{Pt}}}
\newcommand{\SpltExt}{\ensuremath{\mathsf{SpltExt}}}

\newcommand{\GpHSLat}{\ensuremath{\mathsf{GpHSLat}}}
\newcommand{\HSLat}{\ensuremath{\mathsf{HSLat}}}
\newcommand{\Hoops}{\ensuremath{\mathsf{Hoops}}}

\newcommand{\LAb}{\ensuremath{\mathsf{LAb}}}

\newcommand{\Set}{\ensuremath{\mathsf{Set}}}

\DeclareMathOperator{\Ker}{Ker}
\DeclareMathOperator{\Huq}{Huq}
\DeclareMathOperator{\op}{op}
\DeclareMathOperator{\Sub}{Sub}
\DeclareMathOperator{\Eq}{Eq}
\DeclareMathOperator{\Coker}{Coker}
\DeclareMathOperator{\coker}{coker}

\newcommand{\LACC}{{\rm (LACC)}}
\newcommand{\ACC}{{\rm (ACC)}}
\newcommand{\NH}{{\rm (NH)}}
\newcommand{\SH}{{\rm (SH)}}
\newcommand{\SSH}{{\rm (SSH)}}
\newcommand{\W}{{\rm (W)}}

\newcommand{\AC}{{\rm (AC)}}
\newcommand{\ToN}{{\rm (ToN)}}

\newcommand{\Dhead}{-Triangle[open]}
\newcommand{\Dtail}{Triangle[open, reversed]->}

\newdir{>}{{}*:(1,-.2)@^{>}*:(1,+.2)@_{>}}
\newdir{<}{{}*:(1,+.2)@^{<}*:(1,-.2)@_{<}}
\newdir{ >}{% end of arrow for the monos
	@{}*!/-10pt/@{>} }
\newdir{ |>}{% end of arrow for the kernels
	@{}*!/-5.5pt/@{|}*!/-10pt/:(1,-.2)@^{>}*!/-10pt/:(1,+.2)@_{>} }
\newdir{>>}{{}*!/3.5pt/:(1,-.2)@^{>}*!/3.5pt/:(1,+.2)@_{>}*!/7pt/:(1,-.2)@^{>}*!/7pt/:(1,+.2)@_{>}}
\newdir{ >>}{{}*!/8pt/@{|}*!/3.5pt/:(1,-.2)@^{>}*!/3.5pt/:(1,+.2)@_{>}}
\newdir{>}{{}*:(1,-.2)@^{>}*:(1,+.2)@_{>}}
\newdir{<}{{}*:(1,+.2)@^{<}*:(1,-.2)@_{<}}

\begin{document}

\title[Categorical-algebraic aspects of Heyting semilattices]{Categorical-algebraic aspects \\of Heyting semilattices}

\author[García-Martínez]{X.~García-Martínez}
\author[Gray]{J.~R.~A.~Gray}
\author[Hoefnagel]{M.~A.~Hoefnagel}
\author[Van~der Linden]{T.~Van~der Linden}
\author[Vienne]{C.~Vienne}

\address[Xabier García-Martínez]{CITMAga \& Universidade de Santiago de Compostela, Departamento de Mate\-máticas, R\'ua Lope G\'omez de Marzoa s/n, 15782 Santiago de Compostela, Spain}
\email{xabier.garcia@usc.gal}
\address[James R.\ A.\ Gray, Michael A.\ Hoefnagel]{Mathematics Division, Department of Mathematical Sciences, Stellenbosch University, Private Bag X1, 7602 Matieland, South Africa}
\email{jamesgray@sun.ac.za}
\email{mhoefnagel@sun.ac.za}
\address[Tim Van~der Linden, Corentin Vienne]{Institut de
    Recherche en Math\'ematique et Physique, Universit\'e catholique
    de Louvain, che\-min du cyclotron~2 bte~L7.01.02, B--1348
    Louvain-la-Neuve, Belgium}
\email{tim.vanderlinden@uclouvain.be}
\email{corentin.vienne@uclouvain.be}
\address[Tim Van der Linden]{Mathematics \& Data Science, Vrije Universiteit Brussel, Pleinlaan 2, B--1050 Brussel, Belgium}
\email{tim.van.der.linden@vub.be}

\address[Corentin Vienne]{Dipartimento di Matematica ``Federigo Enriques'', Universit\`a degli Studi di Milano, Via Saldini 50, 20133 Milano, Italy}

\thanks{The first author is supported by Ministerio de Economía y Competitividad (Spain), through grant number PID2021-127075NA-I00. The fourth author is a Senior Research Associate of the Fonds de la Recherche Scientifique--FNRS. The fifth author is supported by the \emph{Fonds Thelam} of the \emph{Fondation Roi Baudouin}.}

\begin{abstract}
    This article gives an overview of some key categorical-algebraic properties of the variety of Heyting semilattices, with the aim of correcting a misconception in the literature. We confirm that the category of Heyting semilattices is not \emph{algebraically coherent}, even though it satisfies a strong version of the so-called \emph{Smith is Huq} condition (on the equivalence of two types of commutators).

    We also prove that Higgins commutators of normal subobjects are normal, as a consequence of the fact that Heyting semilattices form an \emph{arithmetical} category. We provide an elementary characterisation of when a pair of subobjects commutes, and use this in the construction of two counterexamples.

    We further show that centralisers exist, centralisers of normal monomorphisms are normal monomorphisms, and normal monomorphisms are closed under composition. We study the latter condition in detail. On the other hand, we show that the category of Heyting semilattices does not satisfy \emph{normality of unions}. Hence, it is not \emph{action accessible} and so it does not admit all normalisers. In particular, this means that the known implication between action accessibility and the condition requiring the existence of centralisers of normal monomorphisms which are themselves normal, is strict.
\end{abstract}

\subjclass[2020]{03G25, 06A12, 18E13}
\keywords{Heyting semilattice, arithmetical semi-abelian category, action accessibility, algebraic coherence, normality of unions, Smith is Huq condition, transitivy of normality, Higgins commutator, centraliser, normaliser}

\maketitle

%%%%%%%%%%%%%%%%%%%%%%%%%%%%%%%%%%%%%%%%%%%%%%%%%%%%%%%%%%%%
\section{Introduction}

\defn{Heyting semilattices}, also known as \defn{implicative} or \defn{Brouwerian} semilattices, form an algebraic structure which is closely related to intuitionistic logic. A~Heyting semilattice \((H,1,\wedge ,\iimplies)\) consists of a set \(H\), a constant \(1\in H\) and two operations \(\wedge\), \(\iimplies \colon H \times H \rightarrow H\) named \defn{meet} and \defn{implication} satisfying the equations
\begin{align*}
    1\wedge x           & = x,                    & (x\iimplies x)           & =1,                                     \\
    x\wedge x           & =x,                     & x\wedge ( x \iimplies y) & = x \wedge y,                           \\
    x\wedge y           & =y\wedge x,             & y \wedge (x \iimplies y) & =y,                                     \\
    x\wedge (y\wedge z) & = (x \wedge y)\wedge z, & x \iimplies (y \wedge z) & = (x\iimplies y) \wedge (x\iimplies z).
\end{align*}
The four identities on the left express that \(H\) is a partially ordered set with a top element \(1\) where \(x \leq y\) if and only if \(x\wedge y = x\). The remaining identities can be thought of as forcing implication to be right adjoint to meet. In fact, as is well-known, an equivalent way of thinking about Heyting semilattices is to define them as \emph{cartesian closed posets}. Indeed, a poset \((H,\leq)\) is a Heyting semilattice precisely when there exists an operation \(\iimplies \colon H \times H \rightarrow H\) such that
\begin{align*}
    x \wedge y \leq z \quad \text{if and only if} \quad x \leq y\iimplies z,
\end{align*}
which exhibits the left adjointness of product (\(\wedge\)) to internal hom (\(\iimplies\)). Therefore, given a poset, there is at most one Heyting semilattice structure on it. Morphisms between Heyting semilattices are defined in the expected way; they preserve the constant and the two operations. Note---see for instance~\cite{Nemitz}---that the following items are true in any Heyting semilattice \(H\), and will be used freely in what follows.
\begin{itemize}
    \item For any \(x\), \(y\), \(z \in H\) we have \((x \wedge y) \iimplies z = x \iimplies (y \iimplies z)=(x\iimplies y)\iimplies (x\iimplies z)\);
    \item for any \(x\in H\) the map \(H\to H\colon y\mapsto (x \iimplies y)\) is a morphism;

    \item for any \(x\), \(y \in H\) we have \(y \leqslant x\iimplies y\) and \(x\leqslant (x \iimplies y) \iimplies y\) by the left adjointness of ``\(\wedge\)'' to ``\(\iimplies\)'', from which we also get \(y \leqslant (x\iimplies y) \iimplies y\).
\end{itemize}

We write \(\HSLat\) for the category of Heyting semilattices with the corresponding morphisms. It is a variety in the sense of Universal Algebra, which has first been shown to be a \emph{semi-abelian} category (in the sense of Janelidze--Márki--Tholen~\cite{Janelidze-Marki-Tholen}) by Johnstone in~\cite{Johnstone:Heyting}, using the following characterisation due to Bourn and Jan\-elidze~\cite{Bourn-Janelidze}:

\begin{theorem}\label{BouJa Charact}
    A variety is semi-abelian if and only if the corresponding algebraic theory contains exactly one constant \(0\) and for some natural number \(n\) we have
    \begin{itemize}
        \item \(n\) binary operations \(\alpha_i\) such that \(\alpha_i(x,x) =0\) for all \(i=1, \dots , n\),
        \item an \((n+1)\)-ary operation \(\Theta\) such that \(\Theta(\alpha_1(x,y), \dots , \alpha_n(x,y),y)= x\).\noproof
    \end{itemize}
\end{theorem}

Johnstone showed that taking the constant \(0\) to be \(1\) and the operations
\begin{align}\label{lemma protomodular terms for HSL}
    \alpha_1(x,y)= x \iimplies y, \;
    \alpha_2(x,y)= ((x\iimplies y)\iimplies y)\iimplies x, \;
    \Theta(x,y,z) = (x \iimplies z) \wedge y,
\end{align}
the conditions of Theorem~\ref{BouJa Charact} are satisfied. What is particularly interesting in Johnstone's paper is that he also shows that it is impossible to characterise Heyting semilattices via semi-abelian terms with \(n=1\). More recently, Lapenta, Metere and Spada proved in~\cite{LMS} that the fact that Heyting semilattices forms a semi-abelian variety can be seen as a corollary of the fact that the variety \(\Hoops\) of the so-called \emph{hoops} is semi-abelian; indeed, \(\HSLat\) is equivalent to the subvariety of idempotent hoops.

It was proven by Köhler~\cite{Kohler} that the variety of Heyting semilattices is such that its lattices of congruences are distributive. In Pedicchio's terminology~\cite{Pedicchio2}, this means that the category \(\HSLat\) is \emph{arithmetical}. Later Rodelo proved in~\cite{Rodelo:Moore} that it is \emph{strongly protomodular} category~\cite{B4} as well, which implies that it satisfies the so-called \emph{Smith is Huq} condition~\cite{BG,MFVdL}. Semidirect products in \(\HSLat\) were described by the authors of~\cite{MMClementinoAMontoliLSousa2015}.

In this paper, we study the categorical-algebraic aspects of Heyting semilattices, most of which are related to commutator theory. While exploring commutators in the context of Heyting semilattices, we understood that some results are generally valid for arithmetical semi-abelian categories. In particular, we prove that in such a category, the Higgins commutator of two normal subobjects is their intersection (Theorem~\ref{theorem arithm commu is intersection}), which implies that it is always a normal subobject again. This shows that any arithmetical semi-abelian category satisfies a condition called (NH)---see Corollary~\ref{Corollary (NH)}.

Another key insight which is elaborated upon in Section~\ref{Section Normal Subobjects} is the fact that since in the category \(\HSLat\), normal subobjects are represented by filters, any composite of two normal monomorphisms is again a normal monomorphism. This has interesting consequences, such as for instance strong protomodularity and the fact that the kernel functors preserve normal closures (Theorem~\ref{theorem ToN consequences}). We analyse this condition in general in semi-abelian categories (Section~\ref{Section ToN})---where we prove it to be equivalent to the condition that any normal subobject is characteristic, and provide a characterisation in terms of the change-of-base functors of the fibration of points. %We come back to the condition in a more general context in Section~\ref{Section (TON) revisited}.

After a more detailed analysis in Section~\ref{section commutators in hslat} of commutativity for pairs of sub-Heyting semilattices, we prove that in \(\HSLat\), centralisers always exist, and centralisers of normal subobjects are normal (Proposition~\ref{Proposition Centralisers}).

What originally motivated us to start our investigations leading to the current article was our discovery of a mistake in a proof in the article~\cite{MFVdL3}, a counterexample involving Heyting semilattices. At first, we wanted to fix the example---until we realised the claim itself is wrong. This led to the results in Sections~\ref{section SSH and W} and~\ref{section AC}, where we (dis)prove certain categorical-algebraic properties for \(\HSLat\). Details are given there; let us just mention that the main results are Theorem~\ref{HSLat is SSH}, Theorem~\ref{HSLat is not W} and Theorem~\ref{HSLat not AC}. Using a single counterexample, we disprove two further categorical-algebraic conditions for \(\HSLat\) in Section~\ref{Section Normalisers and AA}.

We start the mathematical development in Section~\ref{section preliminaries} which recalls some basic features of the categorical context we shall be working in. In Section~\ref{section open questions} we end with some questions which, for now, remain open.

\section{Categorical preliminaries}\label{section preliminaries}

This section is designed in order to provide the (categorical) background on which the next sections are based. The reader might skip it in a first reading and use it for consultation whenever this is required. We recall definitions and results in a manner which is not very detailed, with the main aim of providing a self-contained paper, referencing the literature as much as possible.

\subsection*{Categorical framework}
First, let us sketch the framework in which we are working: semi-abelian categories.

\begin{definition}[\cite{Janelidze-Marki-Tholen}]
    A category \(\C\) is \defn{semi-abelian} when it is pointed, exact (in the sense of Barr), protomodular (in the sense of Bourn) and admits binary coproducts.
\end{definition}

The reader need not be an expert in the subject of semi-abelian categories in order to enjoy the rest of the paper. However, we encourage to keep in mind that it is a weaker framework than abelian categories which allows a unified treatment of homological lemmas, commutator theory or even notions like actions and semidirect products. We strongly recommend the reference book~\cite{Borceux-Bourn} as a guide tour of the subject.

\begin{examples}
    All abelian categories are semi-abelian. But the converse is not true since the category of groups, which constitutes a main example to remember as a motivation, is semi-abelian, as is any \emph{variety of \(\Omega\)-groups} in the sense of Higgins~\cite{Higgins}. Other examples are Heyting semilattices as already mentioned in the introduction, any variety of non-associative algebras over a fixed field \(\K\), the category of rings (with or without unit, morphisms need not preserve the unit when it exists), the dual of the category of pointed objects in an elementary topos, compact Hausdorff groups, etc. Classical non-examples are the category of monoids (protomodularity fails) and the category of topological groups (exactness fails).
\end{examples}

Now, let us recall the following fact about relations in semi-abelian categories. For a detailed account of internal relations, we suggest the Appendix of~\cite{Borceux-Bourn}. Note that from the Barr exactness, it follows that any semi-abelian category has all finite limits.

\begin{proposition}[\cite{Borceux-Bourn}]
    In a semi-abelian category, every reflexive relation is an equivalence relation.\qed
\end{proposition}

Recall that when a category has finite limits and every reflexive relation is an equivalence relation, it is called a \defn{Mal'tsev category}. In Universal Algebra, varieties which form Mal'tsev categories are exactly the ones such that the corresponding algebraic theory contains a \defn{Mal'tsev operation}~\cite{Maltsev-Sbornik}, i.e., a ternary operation \(p\) satisfying \(p(x,y,y)=x\) and \(p(x,x,y)=y\).

For the rest of the paper, unless otherwise stated, the categories we consider are semi-abelian. Even though it may sometimes happen that weaker assumptions can be used, we preferred to present most of our results in a single abstract framework.

\subsection*{Connectors}
The notion of a \emph{connector} is standard in the approach to commutator theory in the context of Mal'tsev varieties~\cite{Smith}, which has been extended to a categorical setting in~\cite{Pedicchio}. Given two equivalence relations
\[
    \begin{tikzcd}
        R \arrow[r, "r_1", shift left = 2] \arrow[r, "r_2"', shift right = 2] & A \arrow[l, "\sigma_R" description]
    \end{tikzcd} \qquad\text{and}\qquad \begin{tikzcd}
        S \arrow[r, "s_1", shift left = 2] \arrow[r, "s_2"', shift right = 2] & A \arrow[l, "\sigma_S" description]
    \end{tikzcd}
\]
on a common object \(A\), we construct the pullback
\begin{center}
    \begin{tikzcd}
        R \times_A S \arrow[r,"p_S", shift left] \arrow[d, shift left,"p_R"] & S \arrow[l, "e_S", shift left] \arrow[d,"s_1", shift left] \\
        R \arrow[r, shift left,"r_2"] \arrow [u,"e_R", shift left] & A \arrow [u,"\sigma_S", shift left] \arrow[l, "\sigma_R", shift left]
    \end{tikzcd}
\end{center}
where \(e_R\) and \(e_S\) are the induced sections.

\begin{definition}[\cite{BG}]
    A \defn{connector} \(p\) on \((R,S)\) is a morphism \({p\colon R \times_A S \rightarrow A}\) such that \(p \circ e_R = r_1\) and \(p\circ e_S=s_2\). When such a connector exists we say that the equivalence relations \(R\) and \(S\) \defn{Smith-commute} or \defn{centralise each other}.
\end{definition}

\begin{remarks}\label{remarks connector}
    \begin{itemize}
        \item In terms of elements, we may view the connector \(p\) as a morphism which satisfies
              \begin{align*}
                  p(a,b,b)=a \quad \text{and} \quad p(a,b,b)=b.
              \end{align*}
              In other words, \(p\) is a partial Mal'tsev operation on \(A\).
        \item In a Mal'tsev category, if two equivalence relations Smith-commute, then the associated connector is necessarily unique.
        \item An internal reflexive graph \begin{tikzcd}[cramped]
                  C_1 \arrow[r, "d", shift left = 2] \arrow[r, "c"', shift right = 2] & C_0 \arrow[l, "e" description]
              \end{tikzcd} can be equipped with a groupoid structure if and only if the kernel pair relations~\(\Eq(c)\) and \(\Eq(d)\) admit a connector. Then, in our context, an internal groupoid structure on a reflexive graph is unique by the previous point. For an account on internal graphs and groupoids, we again recommend the reference~\cite[Appendix~A]{Borceux-Bourn}.
        \item Every equivalence relation on an object \(A\) Smith-commutes with the discrete relation \(\Delta_A\).
    \end{itemize}
\end{remarks}

\subsection*{(Normal) subobjects}
In what follows, we call a subobject \defn{normal} when it is the kernel of some morphism. %More general concepts of normality will arise in Section~\ref{Section (TON) revisited}. 
In order to avoid confusion with the order intrinsic in a Heyting semilattice, we adopt the following:

\begin{convention}
    Given an object \(A\), the order on its class \(\Sub(A)\) of subobjects will be denoted with the symbol ``\(\oleq\)''. Accordingly, meets of subobjects are denoted with the symbol ``\(\omeet\)'' and joins with the symbol ``\(\ojoin\)''. Given a monomorphism \(x\colon X\rightarrowtail A\), the subobject it determines is written in the same way (by abuse of notation) or \(X\oleq A\). We write \(K\normal A\) when \(K\) is a normal subobject of \(A\).
\end{convention}

The normal subobjects in \(\HSLat\) are characterised in Section~\ref{Section Normal Subobjects}.

\subsection*{Cooperators}
In the pointed context of a semi-abelian category, there is an approach to commutators for subobjects essentially due to Huq~\cite{Huq,Borceux-Bourn} which is closer to the classical theory for groups, Lie algebras, etc.

\begin{definition}
    Given two subobjects \(x \colon X \rightarrowtail A\) and \(y\colon Y \rightarrowtail A\), a \defn{cooperator} \(\varphi\) on the pair \((X,Y)\) with respect to \(A\) is a morphism \(\varphi \colon X\times Y \rightarrow A\) such that the following diagram commutes.
    \begin{center}
        \begin{tikzcd}[row sep = large]
            X \arrow[r, "{(1_X,0)}"] \arrow[rd, "x"', tail] & X \times Y \arrow[d, "\varphi" description] & Y \arrow[l, "{(0,1_Y)}"'] \arrow[ld, "y", tail] \\
            & A
        \end{tikzcd}
    \end{center}
    When such a cooperator exists, we say that \(X\) and \(Y\) \defn{cooperate} or \defn{Huq-commute} in \(A\).
\end{definition}

As before, in our context, cooperators are necessarily unique when they exist. In the category of groups, the previous definition works as expected. In fact, it is a quick exercise to prove that two subgroups \(X\) and \(Y\) cooperate if and only if \(xy=yx\) for all \(x\in X\) and \(y\in Y\). An analysis for \(\HSLat\) is the subject of Section~\ref{section commutators in hslat}.

\begin{definition}
    When an object \(X\) Huq-commutes with itself, we call \(X\) an \defn{abelian} object.
\end{definition}

\begin{definition}\label{Def Centraliser}
    Let \(X\oleq A\) be a subobject of an object \(A\). The \defn{centraliser} \(Z_A(X)\) of \(X\) in \(A\) is the largest subobject of \(A\) which cooperates with \(X\).
\end{definition}

\begin{definition}
    The \defn{Huq-commutator} \([X,Y]^{\Huq}\) of \(X\) and \(Y\) in \(A\) is the smallest normal subobject \(N\) of \(A\) such that \(X\) and~\(Y\) cooperate in \({A}/{N}\).
\end{definition}
Given two subobjects \(X\) and \(Y\) of \(A\), we construct the diagram
\begin{center}

    \begin{tikzcd}[row sep = large]
        & X \arrow[dl,"{(1_X,0)}"'] \arrow[d, dotted] \arrow[dr,tail,"x"] & \\

        X\times Y \arrow[r,dotted] & Q & A \arrow[l, dotted,"q"']\\
        & Y \arrow[u,dotted] \arrow[ur,tail,"y"'] \arrow[ul,"{(1_Y,0)}"]
    \end{tikzcd}
\end{center}
where \(Q\) is the colimit of the outer square.

This object \([X,Y]^{\Huq}\) is the kernel of \(q\colon A \rightarrow Q\). The following helps to understand the previous definitions.

\begin{proposition}
    Given two subobjects \(X\) and \(Y\) in \(A\), they admit a cooperator if and only if \([X,Y]^{\Huq}=0\).\qed
\end{proposition}

There exists a strong relationship between cooperators and connectors. In fact, for any equivalence relation \(R\) on \(A\) we can define its \defn{normalisation} \(N_R\) as being the subobject
\begin{center}
    \begin{tikzcd}
        N_R \arrow[r,"{\ker(r_1)}", \Dtail] & R \arrow[r,"r_2"] & A
    \end{tikzcd}
\end{center}
which is normal in \(A\).

\begin{proposition}[\cite{BG}]\label{Smith implies Huq}
    Whenever two equivalence relations \(R\) and \(S\) on a common object admit a connector, then their respective normalisations \(N_R\) and \(N_S\) admit a cooperator.\qed
\end{proposition}

We may ask whether the converse holds: that is, whether
\begin{quote}
    If the normalisations \(N_R\) and \(N_S\) of two equivalence relations \(R\) and \(S\) Huq-commute, then \(R\) and \(S\) Smith-commute.
\end{quote}
In general, this is not the case for all semi-abelian categories, given that the category of digroups and the category of loops are counterexamples~\cite{Borceux-Bourn, HVdL}. This condition which we denote by \SH{} was studied by many authors because of its (co)-homological consequences~\cite{RVdL3,RVdL2}, its impact in the study of internal categorical structures~\cite{MFVdL} and its importance for characterising crossed modules~\cite{Janelidze, HVdL}. The condition \SH{} it is known to be true for Heyting semilattices, since those satisfy \emph{strong protomodularity}---see~\cite{B4,Rodelo:Moore}. In Section~\ref{section SSH and W}, we discuss strengthenings of \SH{} and their validity in \(\HSLat\).

\subsection*{The binary Higgins commutator}

Another commutator for subobjects follows a more universal-algebraic approach: the \emph{Higgins commutator}. It was first defined for \(\Omega\)-groups by Higgins~\cite{Higgins} and then extended to a categorical context by Mantovani and Metere~\cite{MM-NC}.

First, we consider the canonical map
\[
    \Sigma_{X,Y}\coloneq \bigl(\begin{smallmatrix}
            1_X & 0 \\
            0 & 1_Y
        \end{smallmatrix}\bigr)=(\langle1_X,0\rangle,\langle0,1_Y\rangle) \colon X + Y \to X \times Y
\]
for any two objects \(X\) and \(Y\) in a semi-abelian category. The \defn{binary cosmash product} \(X \diamond Y\) is defined as the kernel of \(\Sigma_{X,Y}\). We may check that in the category of groups, \(X \diamond Y\) is generated by \emph{words} in \(X+Y\) of the form \(xyx^{-1}y^{-1}\) with \(x\in X\) and \(y\in Y\). However, those are formal words---elements of \(X+Y\)---rather than elements of some arbitrary group \(A\) containing both \(X\) and \(Y\). This helps to understand the following definition.

\begin{definition}[\cite{MM-NC}]\label{definition higgins commutator}
    Given two subobjects \(x \colon X \rightarrowtail A\) and \(y \colon Y \rightarrowtail A\), there is an induced morphism \(\langle x, y\rangle \colon X+ Y \rightarrow A\). The \defn{Higgins commutator} \([X,Y]\) of \(X\) and \(Y\) in \(A\) is the regular image of \(X\diamond Y\) under the morphism \(\langle x, y\rangle\).

    \begin{center}
        \begin{tikzcd}[column sep = large]
            0 \arrow[r] & X \diamond Y \arrow[r, \Dtail] \arrow[d, \Dhead, dotted] & X+Y \arrow[r, "\Sigma_{X,Y}"] \arrow[d, "{\langle x, y\rangle}"] & X \times Y \\
            & {[X,Y]} \arrow[r, tail, dotted] & A
        \end{tikzcd}
    \end{center}
\end{definition}

\begin{remarks}\label{Remarks Higgins}
    \begin{itemize}
        \item An important relationship, proven in~\cite{MM-NC}, between the Huq commutator and the Higgins commutator is that for any two subobjects \(X\), \(Y\) of an object \(A\), the Huq commutator \([X,Y]^{\Huq}\) is the normal closure of the Higgins commutator \([X,Y]\).
        \item The previous point emphasises that the two commutators are not the same. However, we have that
              \begin{align*}
                  [X,Y]^{\Huq}=0 \quad\Leftrightarrow\quad [X,Y] =0.
              \end{align*}
        \item One could ask whether, in a given semi-abelian category, the Higgins commutator of two normal subobjects is always normal---as it is the case in the category of groups. It turns out that in general, the answer is no: counterexamples were first given by Cigoli in his Ph.D.\ thesis~\cite{AlanThesis}. When this nevertheless happens for a semi-abelian category, we say that the category satisfies \defn{normality of Higgins commutators}. We denote this condition by~\NH. A detailed study of it can be found in~\cite{CGrayVdL1}. In Section~\ref{section arithmetical}, we prove that semi-abelian categories without abelian objects satisfy \NH.
    \end{itemize}
\end{remarks}

It is important for the next sections to mention that the binary cosmash product \(X\diamond Y\) of two objects can be seen as the intersection of \(X \flat Y\) and \(Y\flat X\) in \(X+Y\), where \(X\flat Y\) is defined as the kernel of the morphism \(\langle 1_X,0\rangle\colon X +Y \to X\). This observation is an application of the following well-known lemma.

\begin{lemma}\label{joint kernels lemma}
    In a pointed category with finite limits, we consider \(n\) morphisms \(x_i\colon X \rightarrow Y_i\) where \(i=1, \dots, n\). This induces a morphism
    \[
        (x_1, \dots , x_n)\colon X \rightarrow Y_1 \times \dots \times Y_n.
    \]
    Then the kernel of \((x_1, \dots, x_n)\) is the intersection \(\bigomeet_{i=1}^n \Ker(x_i)\) of the kernels of the~\(x_i\).\qed
\end{lemma}

Since the Higgins commutator \([X,Y]\) of two subobjects \(X\), \(Y \oleq A\) is the regular image of \(X\diamond Y\) under \(\langle x, y\rangle\), it is then natural to ask ourselves what the regular image of \(X\flat Y\) corresponds to. The following answer extends Lemma~6.12 in~\cite{acc}.

\begin{proposition}\label{image of bemol}
    Let \(x\colon X \rightarrowtail A\) and \(y\colon Y\rightarrowtail A\) be two subobjects of an object~\(A\) in a semi-abelian category. Then the regular image of \(X \flat Y\) under the morphism \(\langle x, y\rangle\colon X +Y \rightarrow A\) is the normal closure of \(Y\) in \(X \ojoin Y\), which we denote by \(Y^X\).
\end{proposition}
\begin{proof}
    We consider the commutative diagram
    \begin{center}
        \begin{tikzcd}[column sep = large]
            X \flat Y \arrow[r, \Dtail,"k"] \arrow[d, \Dhead, dotted, "\phi"] & X+Y \arrow[r, "{\langle 1_X,0\rangle}",\Dhead] \arrow[d,\Dhead, "u"] & X \arrow[d,\Dhead,"p"] \\
            L \arrow[r, \Dtail,"l"] & X \ojoin Y \arrow[r,\Dhead, "q"] & Q
        \end{tikzcd}
    \end{center}
    where \(k\) is the kernel of \(\langle 1_X,0\rangle\), \(u\) is the regular epimorphism from the factorisation of \(\langle x, y\rangle\), \((Q,p,q)\) it the pushout of the morphisms \(u\) and \(\langle 1_X,0\rangle\), and \(l\) is the kernel of~\(q\). From those data, we know that \(p\) and \(q\) are regular epimorphisms (pushouts of regular epimorphisms are regular epimorphisms) and that there exists a unique~\(\phi\) making the left hand square commute.

    Applying~\cite[Corollary~4.2]{MM-NC}, we can deduce that \(\phi\) is a regular epimorphism. In other words, \(L\) is the regular image of \(X\flat Y\) under \(u\) and thus under \(\langle x, y\rangle\).

    In order to conclude, it remains to prove that \(q\) is the cokernel of the inclusion \(Y \rightarrowtail X\ojoin Y\) which is routine.
\end{proof}

\begin{corollary}\label{inequality of Higgins}
    Let \(x\colon X \rightarrowtail A\) and \(y\colon Y\rightarrowtail A\) be two subobjects of an object \(A\) in a semi-abelian category. Then we have that
    \begin{align*}
        [X,Y] \oleq X^Y \omeet Y^X.
    \end{align*}
\end{corollary}
\begin{proof}
    The regular image of an intersection is always a subobject of the intersection of the regular images.
\end{proof}

We regain the following characterisation of normality in terms of the Higgins commutator~\cite{MM-NC,HVdL}.

\begin{corollary}\label{Corollary MM-NC}
    Let \(X \oleq A\) be a subobject of an object \(A\) in a semi-abelian category. Then the normal closure \(X^A\) of \(X\) in \(A\) is \(X\ojoin [X,A]\oleq A\), so that \(X\normal A\) if and only if \([X,A]\oleq X\).\noproof
\end{corollary}

\subsection*{Higher Higgins commutators}

In~\cite{Smash,HVdL}, the authors define the \defn{\(n\)-ary cosmash product} \(X_1 \diamond \dots \diamond X_n\) of \(n\) objects \(X_1\), \dots, \(X_n\) as the kernel of the morphism
\begin{align*}
    \Sigma_{X_1,\dots,X_n}\colon X_1+ \dots + X_n \rightarrow \displaystyle \prod_{k=1}^n\coprod_{\substack{j=1 \\j\neq k}}^n X_j
\end{align*}
determined by
\begin{align*}
    \pi_{\coprod_{\substack{j\neq k}}X_j}\circ \Sigma_{X_1,\dots,X_n} \circ \iota_{X_l} = \begin{cases}
                                                                                              \iota_{X_l} & \text{if \(l \neq k\)} \\
                                                                                              0           & \text{if \(l=k\).}
                                                                                          \end{cases}
\end{align*}
One could think that iterating the binary cosmash product suffices to define the \(n\)-ary one. However, in general \((X_1\diamond X_2)\diamond X_3\) is neither isomorphic to \(X_1 \diamond (X_2 \diamond X_3)\), nor to \(X_1\diamond X_2\diamond X_3\). A category where always \((X_1\diamond X_2)\diamond X_3\cong X_1\diamond X_2\diamond X_3\), is said to be \defn{cosmash associative}~\cite{Cosmash}. With Corollary~\ref{Cor Cosmash Assoc} below, we prove that the category of Heyting semilattices does not satisfy this condition.

\begin{definition}\label{Def Higher Higgins}
    Given \(n\) subobjects \(x_i \colon X_i \rightarrowtail A\), there is a universally induced morphism \(\langle x_1, \dots , x_n \rangle \colon X_1 + \dots + X_n \rightarrow A\). The \defn{Higgins commutator} \([X_1, \dots , X_n]\) of \(X_1, \dots , X_n\) in \(A\) is the regular image of \(X_1 \diamond \dots \diamond X_n\) under the morphism \(\langle x_1, \dots , x_n \rangle\).

    \begin{center}
        \begin{tikzcd}[column sep = large]
            0 \arrow[r] & X_1 \diamond \dots \diamond X_n \arrow[r, \Dtail] \arrow[d, \Dhead, dotted] & X_1 + \dots + X_n \arrow[r, "\Sigma_{X_1,\dots ,X_n}"] \arrow[d, "{\langle x_1,\dots , x_n \rangle }"] & \displaystyle \prod_{k=1}^n\coprod_{\substack{j=1 \\j\neq k}}^n X_j \\
            & {[X_1,\dots ,X_n ]} \arrow[r, tail, dotted] & A
        \end{tikzcd}
    \end{center}
\end{definition}

Many basic properties of this commutator can be found in~\cite[Proposition~2.21]{HVdL}. We insist on the fact that, similarly to the \(n\)-ary cosmash product, the \(n\)-ary Higgins commutator is not built up out of iterated binary ones, unless some extra conditions hold: see~\cite{SVdL3}.

We observe that again applying Lemma~\ref{joint kernels lemma}, we deduce that the \(n\)-ary cosmash product \(X_1 \diamond \dots \diamond X_n\) is the intersection
\begin{align*}
    \bigomeet_{i=1}^n ( \coprod_{j\neq i}^n X_j)\flat X_i.
\end{align*}

Variations on the techniques we used before yield:

\begin{proposition}\label{image of n-bemol}
    Given \(n\) subobjects \(x_i\colon X_i \rightarrowtail A\) in a semi-abelian category, the regular image of \(( \coprod_{j\neq i}^n X_j)\flat X_i\) under the morphism \(\langle x_1, \dots, x_n \rangle\) corresponds to \({X_i}^{\bigojoin_k X_k}\), the normal closure of \(X_i\) in \(X_1 \ojoin \cdots \ojoin X_n\).\noproof
\end{proposition}

\begin{corollary}\label{inequality of n-Higgins}
    Given \(n\) subobjects \(x_i\colon X_i \rightarrowtail A\) in a semi-abelian category, we have that \([X_1,\dots ,X_n ]\oleq X_1^{\bigojoin_k X_k} \omeet \cdots \omeet X_n^{\bigojoin_k X_k}\).\noproof
\end{corollary}

\subsection*{Points and split extensions}

Three central notions in semi-abelian categories which will appear throughout this paper are \emph{split extensions}, \emph{points} and \emph{kernel functors}. We fix some notation.

Let \(X\) and \(B\) be objects of a semi-abelian category \(\C\); a \defn{split extension} of \(B\) by \(X\) is a diagram
\begin{equation}\label{eq:split_ext}
    \begin{tikzcd}
        0\ar[r]
        &X\arrow [r, "\kappa"]
        &A \arrow[r, shift left, "\alpha"] &
        B\ar[r]\ar[l, shift left, "\beta"]
        &0
    \end{tikzcd}
\end{equation}
in \(\C\) such that \(\alpha \circ \beta = 1_B\) and \((X,\kappa)\) is the kernel of \(\alpha\). Morphisms of split extensions are morphisms of extensions that commute  with the sections. We write \(\SpltExt_{\C}(X)\) for the category of split extensions in \(\C\) with kernel \(X\). %\todo{Classes of SES? Is the notation standard?}

We write \(\Pt_B(\C)\) for the category of \defn{points} over \(B\) whose objects are triples \((A,\alpha,\beta)\) where \(A\) is an object in \(\C\) and \(\alpha\colon A\rightarrow B\) is a split epimorphism with a given section \(\beta\), and whose morphisms are defined as expected. The functor \(\Ker_B\colon \Pt_B(\C) \to \C\) sending a point \((A,\alpha,\beta)\) to the kernel of \(\alpha\) plays an important role in the literature in general and in this paper. We call this functor the \defn{kernel functor} or \defn{change-of-base functor of the fibration of points}.

\begin{remark}\label{Remark Split Extension}
    In accordance with Theorem~\ref{BouJa Charact}, whenever we have a split extension of Heyting semilattices \eqref{eq:split_ext} as above, any element \(a\) of \(A\) can be written as
    \[
        a=((a\iimplies \beta\alpha(a))\iimplies \beta\alpha(a))\wedge \bigl(((a\iimplies \beta\alpha(a))\iimplies \beta\alpha(a))\iimplies a\bigr)
    \]
    where \(a\iimplies \beta\alpha(a)\) and\(((a\iimplies \beta\alpha(a))\iimplies \beta\alpha(a))\iimplies a\) are in \(X\), while \(\beta\alpha(a)\in B\). In other words, we may write \(a=(x_1\iimplies b)\wedge x_2\) for some \(x_1\), \(x_2\in X\) and \(b\in B\). Thus, viewing \(X\) and \(B\) as subobjects of \(A\), we see that \(A\) is generated by \(X\) and \(B\).
\end{remark}

\section{Arithmetical categories}\label{section arithmetical}

Arithmetical categories were first defined by Pedicchio in~\cite{Pedicchio2} as follows:

\begin{definition}
    An exact Mal'tsev category \(\C\) with coequalisers is said to be \defn{arithmetical} if for any object \(X\) the lattice of equivalence relations on \(X\) is distributive.
\end{definition}

Later, Bourn dropped the existence of coequalisers~\cite{Bourn2001b}. In this paper, when we discuss arithmetical categories we are mainly thinking about categories without abelian objects, also known as \textit{antiadditive} categories in the unital context~\cite{Bourn2001b}. The reason is the following result, which is probably known---for which we could not locate an explicit proof in the literature.

\begin{theorem}\label{theorem arithm no ab obj}
    Let \(\C\) be a semi-abelian category. The following are equivalent:
    \begin{tfae}
        \item \(\C\) is arithmetical,
        \item the category of groupoids in \(\C\) is equivalent to the category of reflexive relations in \(\C\),
        \item the only abelian object of \(\C\) is the zero object.
    \end{tfae}
    \begin{proof}
        (i) {\(\Leftrightarrow\)} (ii): By Pedicchio~\cite{Pedicchio2}.

        (ii) {\(\To\)} (iii): We present an alternative proof to the one given in~\cite{Bourn2001b} (See Remark 3.19). Let \(A\) be an abelian object in \(\C\). This implies that there is a groupoid structure on \(\begin{tikzcd}[cramped]
            A \arrow[r, shift left = 2] \arrow[r, shift right = 2] & 0 \arrow[l]
        \end{tikzcd}\). By hypothesis, we then have that the zero morphism \(A\rightarrow 0\) is jointly monomorphic with itself. This implies that \(A\) is isomorphic to the zero object.

        (iii) {\(\To\)} (ii): Let
        \begin{equation*}
            G=(\begin{tikzcd}[cramped]
                C_1 \arrow[r, "d", shift left = 2] \arrow[r, "c"', shift right = 2] & C_0 \arrow[l, "e" description]
            \end{tikzcd})
        \end{equation*}
        be a groupoid. We want to prove that it is a relation. We know that \(G\) is a reflexive graph such that the kernel pairs \(\Eq(c)\) and~\(\Eq(d)\) commute in the sense of Smith (Remarks~\ref{remarks connector}). Using Proposition~\ref{Smith implies Huq}, this implies that the kernels \(\Ker(c)\) and \(\Ker(d)\) of \(c\) and \(d\) commute in the sense of Huq. By monotonicity of the Huq commutator, we then have that
        \begin{align*}
            [\Ker(c) \omeet \Ker(d), \Ker(c) \omeet \Ker(d)]=0,
        \end{align*}
        which is nothing else than saying that \(\Ker(c)\omeet \Ker(d)\) is an abelian object in \(\C\). By hypothesis, it is then the zero object. This proves that \(c\) and \(d\) are jointly monomorphic (Lemma~\ref{joint kernels lemma}), in other words \(G\) is a relation.
    \end{proof}
\end{theorem}

In Universal Algebra, varieties with distributive lattices of congruences were already studied. The following proposition is a useful characterisation due to Pixley.

\begin{proposition}[\cite{Pixley}]\label{pixley}
    A variety is arithmetical if and only if the corresponding algebraic theory contains a ternary operation \(p\) for which the equations
    \begin{align*}
        p(x,y,y) = x \qquad
        p(x,x,y) = y \qquad
        p(x,y,x) = x
    \end{align*}
    hold.\noproof
\end{proposition}

The first two equations impose \(p\) to be a Mal'tsev operation, while the third one is sometimes known as the \defn{Pixley axiom}.

\begin{examples}
    Using this characterisation, Köhler concludes in~\cite{Kohler} that the variety \(\HSLat\) is arithmetical, since the ternary operation \(p\) defined as
    \begin{align*}
        p(x,y,z) = ((x\iimplies y) \iimplies z) \wedge ((z\iimplies y)\iimplies x )\wedge ((z\iimplies x)\iimplies x)
    \end{align*}
    satisfies the equations of Proposition~\ref{pixley}. Other examples of arithmetical categories are the category of boolean rings without unit, the category of von Neumann regular rings or the dual of any topos (see~\cite{Borceux-Bourn}).
\end{examples}

\subsection*{Higgins commutators in arithmetical context}

As we recalled above, in~\cite{Pedicchio}, Pedicchio defined a commutator \([R,S]\) for two arbitrary equivalence relations \(R\) and~\(S\) on an object \(X\) in an exact Mal'tsev category with coequalisers. Her definition generalises the one from Smith in the context of Mal'tsev varieties~\cite{Smith}, which explains why it is often referred to as the \defn{Smith--Pedicchio commutator}. Later, in~\cite{Pedicchio2}, not only she defines arithmetical categories, but she also characterises them as follows:
\begin{theorem}\label{Pedicchio charact of arithm}
    Let \(\C\) be an exact Mal'tsev category with coequalisers. The following are equivalent:
    \begin{tfae}
        \item \(\C\) is arithmetical,
        \item for any equivalence relations \(R\), \(S\) on a common object, \([R,S]=R\omeet S\).\noproof
    \end{tfae}
\end{theorem}

Here, our goal is similar but instead of studying commutators of equivalence relations, we focus on Higgins commutators and move to a semi-abelian context. First we need the following lemma.

\begin{lemma}\label{image of intersections}
    In a semi-abelian category without abelian objects, let \(K\), \(L  \normal A \) be two normal subobjects of \(A\) and \(f\colon A \rightarrow B\) a morphism. Then we have that
    \begin{align*}
        f(K)\omeet f(L) = f(K\omeet L).
    \end{align*}
\end{lemma}
\begin{proof}
    We write \(f=m\circ e\) where \(m\) is a monomorphism and \(e\) a regular epimorphism.
    Applying~\cite[Theorem~2.1]{Bourn:Direct-Image}, we know that the quotient~\(e(K)\omeet e(L) \Big/ e(K\omeet L)\) is an abelian object. However, abelian objects are trivial by hypothesis, so that \(e(K)\omeet e(L) = e(K\omeet L)\). Now
    \[
        f(K)\omeet f(L) = me(K)\omeet me(L) = m(e(K)\omeet e(L)) = me(K\omeet L) = f(K\omeet L)\text{,}
    \]
    which concludes the proof.
\end{proof}

This allows us to prove without any effort a characterisation for arithmetical categories in a semi-abelian context.

\begin{theorem}\label{theorem arithm commu is intersection}
    Let \(\C\) be a semi-abelian category. The following are equivalent:
    \begin{tfae}
        \item The only abelian object of \(\C\) is the zero object.
        \item Let \(X\), \(Y \oleq A\) be two subobjects of \(A\) in \(\C\); the Higgins commutator of \(X\) and \(Y\) is \(X^Y \omeet Y^X\).
        \item Every object \(A\) in \(\C\) is \defn{perfect}, which means that \([A,A]=A\).
    \end{tfae}
\end{theorem}
\begin{proof}

    (i) {\(\To\)} (ii): Let \(x\colon X\to A\) and \(y\colon Y\to A\) be two monomorphisms. Knowing that \(X\diamond Y= X\flat Y \omeet Y\flat X\normal X+Y\), we simply apply Lemma~\ref{image of intersections} with \(K\) and~\(L\) being respectively \(X\flat Y\) and \(Y\flat X\) viewed as subobjects of \(X+Y\) and then consider their direct image along the morphism \(\langle x,y\rangle\colon X+Y\to A\). The result now follows from Proposition~\ref{image of bemol}.

    (ii) \(\To\) (iii): For any object \(A\) we have that \([A,A]=A^A \omeet A^A=A\).

    (iii) \(\To\) (i): Considering an abelian object \(A\), we have that \(A=[A,A]=0\).
\end{proof}

\begin{example}
    Using a similar \textit{Eckmann-Hilton} argument as in Corollary~\ref{corollary hslat no abelian objects}, one can see that the category \(\Hoops\) has no abelian objects and thus is arithmetical.
\end{example}

We observe that if in Theorem~\ref{theorem arithm commu is intersection} we take \(X\) and \(Y\) to be normal subobjects, then their Higgins commutator is their intersection in \(A\). Moreover, since the intersection of two normal subobjects is again a normal subobject, we have in this case that the commutator of Higgins is the Huq commutator (see Section~\ref{section preliminaries}). In other words, we have the following, which extends Proposition~2.2 in~\cite{GrayVdL1}:

\begin{corollary}\label{Corollary (NH)}
    A semi-abelian category without abelian objects satisfies \NH{}.\noproof
\end{corollary}

The same type of reasoning allows us to conclude the following even outside of the arithmetical context.

\begin{proposition}\label{Proposition Normal Closure}
    For a perfect object \(X\) in a semi-abelian category, whenever \({X\oleq A}\) we have that \(X^A=[X,A]\).
\end{proposition}
\begin{proof}
    \(X^A=X\ojoin [X,A]=[X,X]\ojoin [X,A]=[X,A]\) by Corollary~\ref{Corollary MM-NC}.
\end{proof}

Since the \(n\)-ary cosmash product, for \(n\geq 2\), can also be seen as an intersection---see Section~\ref{section preliminaries}---applying Proposition~\ref{image of n-bemol} and similar techniques as in the proof of Theorem~\ref{theorem arithm commu is intersection}, we conclude:

\begin{theorem}\label{theorem arithm n-comm}
    Let \(\C\) be a semiabelian category. The following are equivalent:
    \begin{tfae}
        \item The only abelian object of \(\C\) is the zero object.
        \item For any \(n\geq 2\), given \(n\) subobjects \(X_i \oleq A\), the inequality of Corollary~\ref{inequality of n-Higgins} is an equality:
        \begin{align*}
            [X_1,\dots ,X_n ] = X_1^{\bigojoin_k X_k} \omeet \cdots \omeet X_n^{\bigojoin_k X_k}.
        \end{align*}
        \item For any \(n\geq 2\), every object \(A\) in \(\C\) is \defn{\(n\)-perfect}, which means that the \(n\)-ary Higgins commutator \([A,\dots, A]\) is \(A\).\noproof
    \end{tfae}
\end{theorem}

Morally, this means that, in a semiabelian context, arithmetical categories can be seen as the ones with the largest Higgins commutators possible, so that \(A=[A,A]=[A,A,A]=[A,\dots,A]\) for any object \(A\). Since~\cite{Pedicchio2} the Smith--Pedicchio commutator \([R,S]\) is always smaller than \(R\omeet S\), Theorem~\ref{theorem arithm n-comm} can be understood in the spirit of Theorem~\ref{Pedicchio charact of arithm}.

\begin{corollary}\label{corollary arithm NHn}
    Let \(\C\) be a semiabelian category without abelian objects and \(n\geq 2\) a natural number. Then \(\C\) satisfies the condition
    \begin{quote}
        For any \(n\) normal subobjects \(K_i  \normal A\), the \(n\)-ary Higgins commutator \([K_1,\dots, K_n]\) is a normal subobject of \(A\).
    \end{quote}
    which we denote \(\NH^n\).\noproof
\end{corollary}

We propose to name the condition \(\NH^n\) the \defn{\(n\)-normality of Higgins commutators}. It is clear that for \(n=2\) we recover the aforementioned \emph{normality of Higgins commutators}. We remark that the converse of the previous corollary is not true, since the category \(\Grp\) of groups satisfies \NH{}, while of course non-trivial abelian objects exist (namely, the abelian groups). %Open questions about \(\NH^n\) are mentioned in Section~\ref{section open questions}.

\section{Normal subobjects}\label{Section Normal Subobjects}
In order to describe the normal subobjects in \(\HSLat\), we recall the definition of a filter.

\begin{definition}
    A non-empty subset \(K\) of a Heyting semilattice \(A\) is a \defn{filter} of \(A\) when
    \begin{enumerate}
        \item if \(x\), \(y\in K\) then \(x\wedge y \in K\);
        \item if \(x\in K\) and \(y\in A\) such that \(y\geq x\), then \(y\in K\).
    \end{enumerate}
\end{definition}

\begin{example}
    In a \emph{finite} Heyting semilattice \(A\), any filter is of the form
    \begin{equation*}
        \ua{x}=\lbrace y \in A \mid y\geq x\rbrace
    \end{equation*}
    for some \(x\in A\). Let, indeed, \(K\) be a filter in \(A\); we consider the element \(x=\bigwedge_{k\in K} k\), which exists because \(K\) is finite. By definition of filters, we have that \(x\in K\) and that \(\ua{x} \oleq K\). But since \(x\) is lower than or equal to any element in \(K\), we have \(K\oleq \ua{x}\) and thus \(K=\ua{x}\).
\end{example}

It was shown by Nemitz in \cite{Nemitz} that:

\begin{lemma}\label{Lemma Normal is Filter}
    A non-empty subset \(K\) of a Heyting semilattice \(A\) is the kernel of some morphism \(f\colon A\rightarrow B\) of Heyting semilattices if and only if~\(K\) is a filter. In categorical terminology, filters correspond to \defn{normal subobjects} in the category~\(\HSLat\). \noproof
\end{lemma}
Therefore, for any filter of \(A\) of the form \(\ua{x}\) we can define its quotient \(A/\ua{x}\). In the finite case, this quotient is isomorphic to \(\da{x}\), i.e., the Heyting semilattice formed by all the elements lower than \(x\), with \(x\) playing the role of \(1\). This object is further isomorphic to the Heyting semilattice \(x \iimplies A\coloneq \{ x\iimplies a \mid a \in A \}\).

\begin{lemma}\label{lemma normal closure finite}
    Let \(X\) be a subobject of a finite Heyting semilattice \(A\). Then the normal closure of \(X\) in \(A\) is \(\ua{a}\) where \(a=\bigwedge_{x\in X} x\).
\end{lemma}
\begin{proof}
    It is clear that \(X\oleq \ua{a}\). To see that \(\ua{a}\) is the smallest filter in~\(A\) containing~\(X\), let \(K\) be a filter of \(A\) such that \( X\oleq K\). Then \(K=\ua{k}\) for some \(k\in A\). Since \(a\in X\) and \(X\oleq K\), then \(k\leq a\), which implies that \(\ua{a} \oleq \ua{k}\).
\end{proof}

The interpretation of normal monomorphisms as filters allows us to make the following key observation:

\begin{proposition}
    In the category \(\HSLat\), a composite of two normal monomorphisms is again a normal monomorphism.
\end{proposition}
\begin{proof}
    The claim follows easily from Lemma~\ref{Lemma Normal is Filter} by noting that if a subset \(X\) of a Heyting semilattice \(A\) is up-closed in another subset \(Y\) of \(A\), and \(Y\) is up-closed in~\(A\), then \(X\) is up-closed in \(A\).
\end{proof}

We say that the category of Heyting semilattices satisfies \emph{transitivity of normality}. After investigating a few further examples, we spend the next section on deducing consequences.

\begin{definition}
    We say that in a semi-abelian category, \defn{normality is transitive} or that it satisfies \defn{transitivity of normality (ToN)} when any composite of two normal monomorphisms is again a normal monomorphism.
\end{definition}

\begin{remark}
    Since in a semi-abelian category, normal epimorphisms are always closed under composition~\cite{Borceux-Bourn}, the condition \ToN{} is equivalent to axiom (ex2) in~\cite{Grandis-HA2}. Since Grandis's axioms (ex0), (ex1) and (ex3) all hold in any semi-abelian category (with respect to the ideal of zero maps), a semi-abelian category satisfies \ToN{} if and only if it is \emph{Grandis-homological}~\cite{Grandis-HA2}.
\end{remark}

\begin{examples}
    \begin{enumerate}
        \item Any abelian category satisfies \ToN, because in it all mono\-morphisms are normal.
        \item The arithmetical semi-abelian category \(\Set^{\op}_*\), the dual of the category of pointed sets, satisfies \ToN{} because in \(\Set_*\), a morphism is a normal epimorphism precisely when it is injective outside the inverse image of the base point; it is clear that a composite of two normal epimorphisms still satisfies this property.
        \item The category \(\Hoops\) satisfies \(\ToN\) for essentially the same reason \(\HSLat\) does: a subhoop is normal if and only if it is a filter~\cite{BlokPigozzi}.
        \item In the category \(\LAb\) of lattice-ordered abelian groups, normal subobjects correspond to convex subgroups. Therefore, normality is transitive in \(\LAb\).
    \end{enumerate}
\end{examples}

\begin{remark}
    The condition \ToN{} makes sense outside of the semi-abelian context: for instance, in~\cite{JQT} it is explained that the category of commutative monoids and the category of topological abelian groups satisfy this condition, and in~\cite{PVdL2} a link with the validity of the Snake Lemma is explored. In the present article, we focus on semi-abelian categories.
\end{remark}

Non-abelian semi-abelian examples are scarce. Most are easily checked by means of the next lemma.

\begin{lemma}\label{Lemma ToN Char}
    For a perfect object \(X\) in a semi-abelian category, the following are equivalent:
    \begin{tfae}
        \item if \(X\normal Y\normal A\) then \(X\normal A\);
        \item whenever \(X\oleq A\) we have \(X^A=X^{(X^A)}\);
        \item whenever \(X\oleq A\) we have \([X,A]= [X,[X,A]]\);
        \item whenever \(X\oleq A\) we have \([X,A]\oleq [X,[X,A]]\);
        \item whenever \(X\oleq A\) we have \([[X,X],A]\oleq [X,[X,A]]\).
    \end{tfae}
\end{lemma}
\begin{proof}
    (i) \(\To\) (ii) because if (i) holds and \(X\oleq A\) then \(X^{(X^A)}\normal X^A\normal A\) implies \(X^{(X^A)}\normal A\) so that \(X^{(X^A)}=X^A\) and (ii) holds. (ii) \(\To\) (iii) depends on Proposition~\ref{Proposition Normal Closure}, which implies that
    \begin{align*}
        [X,A] =X^A=X^{(X^A)}= [X,X^A]= [X,[X,A]].
    \end{align*}
    (iii) \(\To\) (iv) is obvious, and (iv) \(\Leftrightarrow\) (v) holds by the assumption that \(X=[X,X]\). Now (iv) \(\To\) (i) by Corollary~\ref{Corollary MM-NC}, since \(X\normal Y\normal A\) implies \([X,Y]\oleq X\) and \([Y,A]\oleq Y\), so we may write
    \begin{align*}
        [X,A] \oleq [X,[X,A]] \oleq [X,[Y,A]]\oleq [X,Y]\oleq X
    \end{align*}
    and conclude that \(X\normal A\).
\end{proof}

\begin{examples}
    \begin{enumerate}
        \item The variety of Boolean rings being arithmetical, all of its objects are perfect. The lemma now allows us to conclude that \ToN{} holds, because if \(X\oleq A\), then we have \([[X,X],A]\oleq [X,[X,A]]\) since by associativity, \((xx')a\in [[X,X],A]\) is equal to \(x(x'a)\in [X,[X,A]]\).
        \item  The condition \ToN{} does not hold in the category of groups: we have \(\{e,(12)(34)\}\normal\{e,(12)(34),(13)(24),(23)(14)\}\normal S_4\) but \(\{e,(12)(34)\}\not\normal S_4\).
        \item A variety of non-associative algebras over an infinite field \(\K\) does not satisfy \ToN{} in general, unless it is the abelian or trivial variety~\cite{CocoThesis}. However, for an \emph{algebraically coherent} variety (see Section~\ref{section AC}) over an infinite field, transitivity of normality holds when the smallest subobject is perfect. This can be seen by using Theorem 2.12 of~\cite{GM-VdL2} and the fifth point of Lemma~\ref{Lemma ToN Char}.
    \end{enumerate}
\end{examples}

\begin{remark}
    The two conditions of arithmeticity and transitivity of normality are separated. In fact, on one hand we have abelian categories as counterexamples. On the other hand, the category \(\GpHSLat\) of groups equipped with a Heyting semilattice structure such that both structures share the same unit \(e\) is arithmetical (because of Proposition~\ref{pixley}) but not \ToN. This can be seen by considering the symmetric group \(S_4\), in which normality of subgroups is not transitive, and equipping its underlying set with a total order where \((13)(24)<(23)(14)<(12)(34)<e\) are the four greatest elements.
\end{remark}

\begin{remark}
    An example of a semi-abelian \ToN{} category that is neither abelian nor arithmetical is the product of the category \(\Ab\) of abelian groups with the category \(\HSLat\) of Heyting semilattices.
\end{remark}

\section{Categorical-algebraic consequences of transitivity of normality}\label{Section ToN}

We study some consequences of the condition \ToN{} in the context of semi-abelian categories. We start with some lemmas and a definition.

\subsection*{Some lemmas}
Let us consider two categories \(\C\) and \(\D\) such that \(\C\) admits finite limits.

\begin{lemma}
    \label{reflects_order}
    Let \(F\colon\C\to \D\) be a functor which preserves limits and reflects isomorphisms. For each object \(A\) of \(\C\), the induced map
    \begin{align*}
        \Sub(A)\to \Sub(F(A))  \colon (X\rightarrowtail A) \mapsto (F(X)\rightarrowtail F(A))
    \end{align*}
    (which is well defined because of the assumptions) is order preserving, preserves meets and is order reflecting.
\end{lemma}
\begin{proof}
    First, since \(F\) preserves limits, it is a direct consequence that the induced map preserves the order and meets. Now, given two subobjects \(X\) and \(Y\) of \(A\), if \(F(X) \oleq F(Y)\) then we have that
    \begin{align*}
        F(X \omeet Y) =F(X) \omeet F(Y) =F(X).
    \end{align*}
    Since \(F\) reflects isomorphisms, this implies that \(X\omeet Y =X\) and thus \(X\oleq Y\).
\end{proof}

Let us now suppose that moreover \(\C\) and \(\D\) are pointed categories.

\begin{definition}
    A functor \(F\colon \C \rightarrow \D\) \defn{creates normal subobjects} if for any object \(A\) in \(\C\) and any normal subobject \(k'\colon K' \rightarrowtail F(A)\) in \(\D\) there exists a unique normal subobject \(k \colon K \rightarrowtail A\) in \(\C\) such that \(F(k)=k'\).

\end{definition}

\begin{lemma}\label{reflects_normal}
    Let \(F\colon\C\to \D\) be a functor which preserves limits and reflects isomorphisms. If \(F\) creates normal subobjects, then \(F\) reflects normal monomorphisms.
\end{lemma}
\begin{proof}
    Suppose \(x\colon X\to A\) is a morphism in \(\C\) such that \(F(x)\colon F(X) \rightarrowtail F(A)\) is a normal monomorphism. Since \(F\) creates normal subobjects, it follows that there exists a unique normal subobject \(k\colon K \rightarrowtail A\) such that \(F(k)=F(x)\). We conclude that \(x\) is a monomorphism and \(x=k\) from the fact that \(F\) reflects isomorphisms and preserves limits (which implies that \(F\) is faithful).
\end{proof}

\begin{lemma}\label{Lemma Preserves Normal Closure}
    Let \(F\colon\C\to \D\) be a functor which preserves limits and reflects isomorphisms. If \(\D\) admits normal closure of monomorphisms and \(F\) creates normal subobjects, then \(\C\) admits normal closure of monomorphisms and \(F\) preserves them.
\end{lemma}
\begin{proof}
    We start by proving that \(\C\) admits normal closure. Suppose \(x\colon X\rightarrowtail A\) is a monomorphism in \(\C\), then \(F(x) \colon F(X) \rightarrowtail F(A)\) is a monomorphism in \(\D\). We consider \(\overline{F(x)}\colon \overline{F(X)}\rightarrowtail F(A)\), the normal closure of \(F(x)\). Because \(F\) creates normal subobjects, there exists a unique \(\overline{x}\colon \overline{X}\rightarrowtail A\), normal subobject in \(\C\), such that \(F(\overline{x})=\overline{F(x)}\). We want to show that \(\overline{x}\) is the normal closure of \(x\). Since \(F(X) \oleq \overline{F(X)} = F(\overline{X})\), it follows from Lemma~\ref{reflects_order} that \(X\oleq \overline{X}\). Now that we proved that \(X\) is included in the normal subobject \(\overline{X}\), we need to prove that \(\overline{X}\) is the smallest normal subobject of \(A\) containing \(X\). Let \(k\colon K \rightarrowtail A\) be a normal subobject such that \(X\oleq K\). Then \(F(X) \oleq F(K)\) and thus \(F(\overline{X}) = \overline{F(X)}\oleq F(K)\). Using again Lemma~\ref{reflects_order}, we conclude that \(\overline{X}\oleq K\).
    Since normal closure is unique we see that \(F\) preserves normal closure by construction.
\end{proof}

% \begin{lemma}\label{Lemma preserves joins}
%     Let \(F\colon\C\to \D\) be a functor which preserves limits, reflects isomorphisms and creates normal subobjects. If binary joins of normal subobjects exist in \(\D\) and are normal, then binary joins of normal subobjects exist in \(\C\) and are normal. Moreover, \(F\) preserves binary joins of normal subobjects.
% \end{lemma}
% \begin{proof}
%     Given two normal subobjects \(k\colon K\rightarrowtail A\) and \(l\colon L\rightarrowtail A\) in \(\C\), we then have that \(F(k)\colon F(K)\rightarrow F(A)\) and \(F(l)\colon F(L)\rightarrowtail F(A)\) are normal subobjects in \(\D\) and their join \(F(k)\ojoin F(l)\colon F(K)\ojoin F(L) \rightarrow F(A)\) is normal too. Since \(F\) creates normal subobjects, there exists a unique normal subobject \(m\colon M\rightarrowtail A\) such that \(F(M)=F(K)\ojoin F(L)\). From here, it is clear that \(F(K)\oleq F(M)\) and thus \(K\oleq M\) by Lemma~\ref{reflects_order}. For the same reason, we have that \(L\oleq M\). Given another subobject \(n\colon N \rightarrow A\) such that \(K\oleq N\) and \(L\oleq N\), then we have that \(F(K)\oleq F(N)\) and \(F(L)\oleq F(N)\). This implies that \(F(M)=F(K)\ojoin F(L) \oleq F(N)\). By Lemma~\ref{reflects_order}, it gives us that \(M\oleq N\) which proves that \(M\) is the join of \(K\) and \(L\).
% \end{proof}

\subsection*{Consequences}

Toward presenting the main results of this section (Theorems~\ref{Thm Char ToN} and~\ref{theorem ToN consequences}), we first recall the notion of a \emph{characteristic} subobject. Introduced by A.~S.~Cigoli and A. Montoli in the context of semi-abelian categories~\cite{CigoliMontoliCharSubobjects}, this concept extends the classical notion of a characteristic subgroup (or a characteristic ideal in Lie algebras) as follows:

\begin{definition}\cite{CigoliMontoliCharSubobjects}\label{definition charact subobj}
    In a semi-abelian category, a subobject \(h\colon H \rightarrowtail G\) is characteristic in \(G\) if and only if, for every split extension of \(B\) by \(G\)
    \[
        \xymatrix{
            G \ar[r]^{k} & A \ar@<.5ex>[r]^{\alpha} & B \ar@<.5ex>[l]^{\beta}
        }
    \]
    there exists a split extension of \(B\) by \(H\)
    \[
        \xymatrix{
            H \ar[r]^{k'} & A' \ar@<.5ex>[r]^{\alpha'} & B \ar@<.5ex>[l]^{\beta'}
        }
    \]
    and a morphism of split extensions as follows:
    \[
        \xymatrix{
            H \ar[r]^{k'} \ar[d]_{h} & A' \ar@<.5ex>[r]^{\alpha'} \ar[d]_{h'} & B \ar@<.5ex>[l]^{\beta'} \ar@{=}[d] \\
            G \ar[r]^{k} & A \ar@<.5ex>[r]^{\alpha} & B \ar@<.5ex>[l]^{\beta}
        }
    \]
\end{definition}

The next proposition establishes that, as in the category of groups, ``is a characteristic subobject of'' is a transitive relation.

\begin{proposition}\cite{CigoliMontoliCharSubobjects}\label{proposition charact subobj transitive}
    In a semi-abelian category, if \(H\) is characteristic in \(K\) and \(K\) is characteristic in \(G\), then \(H\) is characteristic in \(G\).
    \noproof
\end{proposition}

``Being a characteristic subobject'' is stronger than ``being a normal subobject'', as states the next proposition.

\begin{proposition}\cite{CigoliMontoliCharSubobjects}\label{proposition charact implies normal}
    In a semi-abelian category, if \(H\) is a characteristic subobject of \(G\), then \(H\) is a normal subobject of \(G\).
    \noproof
\end{proposition}

The converse of this proposition is false, since the category of groups provides a counterexample. However, in any semi-abelian category satisfying~\ToN, the two classes of subobjects do coincide.

\begin{proposition}\label{Proposition Semiabelian (TON)}
    Let \(\C\) be a semi-abelian category in which normality is transitive. Then in \(\C\), every normal subobject is characteristic.
\end{proposition}
\begin{proof}
    Suppose that \(h \colon H \rightarrowtail G\) is a normal subobject, and consider a split extension \((k, A, \alpha, \beta)\) of \(B\) by \(G\). Since \(\C\) satisfies \ToN, it follows that \(k \circ h\) is a normal monomorphism. By Lemma~2.6 of~\cite{CGrayVdL1}, there exists a split extension \((k', A', \alpha', \beta')\) of \(B\) by \(H\) and a morphism of split extensions as described above.
\end{proof}

\begin{theorem}\label{Thm Char ToN}
    For a semi-abelian category \(\C\), the following conditions are equivalent:
    \begin{tfae}
        \item \(\C\) satisfies \ToN;
        \item each normal subobject is characteristic;
        \item for each object \(B\) in \(\C\), the kernel functor \(\Ker_B \colon \Pt_B(\C) \rightarrow \C\) creates normal subobjects.
    \end{tfae}
\end{theorem}
\begin{proof}
    (i) implies (ii) by Proposition~\ref{Proposition Semiabelian (TON)}. On the other hand, applying Proposition~\ref{proposition charact subobj transitive} proves that (ii) implies (i). The well-known fact (Chapter 6 of~\cite{Borceux-Bourn}) that in the diagram of the proof of Proposition~\ref{Proposition Semiabelian (TON)}, the monomorphism \(h'\) considered in \(\Pt_B(\C)\) is normal if and only if \(k\circ h\) is normal in \(\C\) shows that (ii) and (iii) are equivalent.
\end{proof}

\begin{theorem}\label{theorem ToN consequences}
    Let \(\C\) be a semi-abelian category in which normality is transitive. We then have that
    \begin{enumerate}
        \item for any object \(B\) in \(\C\), the kernel functor \(\Ker_B \colon \Pt_B(\C) \rightarrow \C\) preserves normal closures;
        \item the category \(\C\) is strongly protomodular;
        \item the category \(\C\) satisfies \SH;
        \item the category \(\C\) satisfies \NH.
    \end{enumerate}
\end{theorem}
\begin{proof}
    The first statement is a consequence of combining Proposition~\ref{Proposition Semiabelian (TON)} and Lemma~\ref{Lemma Preserves Normal Closure}.

    For the second statement, we get from Lemma~\ref{reflects_normal} that kernel functors reflect normal monomorphisms, which (in the current Barr-exact context) is \defn{strong protomodularity} by definition---see~\cite{B4}.

    It was proven in~\cite{BG} that any pointed category which is strongly protomodular satisfies the condition \SH. Therefore, the third statement is consequence of the second one.

    We now prove the fourth statement. Let \(K\), \(L  \normal A \) be two normal subobjects of an object \(A\) in \(\C\). From~\cite{MM-NC}, it is known that the Higgins commutator \([K,L]\) is a normal subobject of the join \(K\cup L\). It is known that in semi-abelian categories, binary joins of normal subobjects are normal---see for instance~\cite{EverVdLRCT} for an explicit proof. Therefore, using transitivity of normality, we conclude that \([K,L]\) is normal in \(A\).
\end{proof}

Let us add a characterisation in terms of preservation of normal monomorphisms by pushouts along regular epimorphisms:

\begin{theorem}\label{Thm ToN via Pushouts}
    For a semi-abelian category \(\C\), the following conditions are equivalent:
    \begin{tfae}
        \item \(\C\) satisfies \ToN;
        \item in \(\C\), any pushout of a normal monomorphism along a regular epimorphism is again a (necessarily normal) monomorphism.
    \end{tfae}
\end{theorem}
\begin{proof}
    Consider two normal monomorphisms \(\kappa\) and \(v\), a cokernel \(\alpha\) of \(\kappa\), and a pushout of \(v\) and \(\alpha\) as in the diagram
    \[
        \xymatrix{
        X \ar[r]^-{\kappa}\ar@{-->}[d]_{u} & A \ar[r]^{\alpha}\ar[d]^{v} & B \ar[d]^-{w}\\
        X' \ar@{-->}[r]_-{\kappa'} & A' \ar[r]_-{\alpha'} & A'+_A B
        }
    \]
    If (ii) holds then \(w\) is a monomorphism, so that the induced square between the kernels on the left is a pullback~\cite[Lemma~4.2.4]{Borceux-Bourn} and \(X=X'\omeet A\). It follows that the inclusion of \(X\) into \(A'\), the composite \(v\kappa\), is a normal monomorphism, so that (i) holds.

    For the converse, assume that \(v\kappa\) is a normal monomorphism. Then the quotient \(A'/X=\Coker(v\kappa)\) is a pushout of \(v\) and \(\alpha\), as is easily verified by hand. So we may identify the arrow \(\coker(v\kappa)\) with \(\alpha'\), while \(\kappa'=v\kappa\) and \(u=1_X\). Thus, this pushout square is a pullback~\cite[Lemma~4.2.5]{Borceux-Bourn}, so that the induced morphism \(w\) is a monomorphism~\cite[Lemma~4.2.5]{Borceux-Bourn}. It is a normal monomorphism as a direct image of a kernel along a regular epimorphism in a semi-abelian category~\cite{Janelidze-Marki-Tholen}.
\end{proof}

All of the results, obtained in this section, can be applied to the specific case of Heyting semilattices as a corollary. As mentioned in the introduction, the strong protomodularity of \(\HSLat\) was proved in~\cite{Rodelo:Moore}; while the coincidence of Higgins and Huq commutators on normal subobjects was already explained in Proposition~\ref{Corollary (NH)}. The advantage of the present approach is that it works in any \ToN{} semi-abelian category.

\section{Commuting subobjects}\label{section commutators in hslat}
We do not have an explicit description of the commutators in \(\HSLat\), but we can characterise when two given subobjects commute in terms of the implication operation.

\begin{lemma}\label{lemma commute in hslat}
    Let \(X\), \(Y \oleq A\) be two subobjects of a Heyting semilattice \(A\). Then \(X\) and \(Y\) Huq-commute in \(A\) if and only if for all \(x\in X\) and \(y\in Y\) we have that
    \begin{align*}
        (x\iimplies y) = y \qquad \text{and} \qquad (y\iimplies x) =x.
    \end{align*}
\end{lemma}
\begin{proof}
    We first suppose that \(X\) and \(Y\) Huq-commute in \(A\). This, by definition, means that there exists a unique morphism \(\varphi\colon X \times Y \rightarrow A\) making the following diagram commute.
    \begin{center}
        \begin{tikzcd}[row sep = large]
            X \arrow[r, "{(1_X,0)}"] \arrow[rd, tail] & X \times Y \arrow[d, "\varphi" description] & Y \arrow[l, "{(0,1_Y)}"'] \arrow[ld, tail] \\
            & A
        \end{tikzcd}
    \end{center}
    This implies that \(\varphi(x,y)=x\wedge y\) for all \((x,y)\in X \times Y\), because indeed
    \begin{align*}
        \varphi(x,y)=\varphi((x,1)\wedge (1,y)) = \varphi(x,1)\wedge \varphi(1,y) = x\wedge y.
    \end{align*}
    Next, we have that for all \(x\), \(x'\in X\) and \(y\), \(y'\in Y\)
    \begin{align*}
        \varphi((x,y)\iimplies (x',y'))= \varphi(x\iimplies x', y \iimplies y') = (x\iimplies x')\wedge (y\iimplies y')
    \end{align*}
    and
    \begin{align*}
        \varphi(x,y)\iimplies \varphi(x',y') & = (x\wedge y) \iimplies (x'\wedge y') \\ &= ((x\wedge y)\iimplies x')\wedge ((x\wedge y)\iimplies y') \\ &=(x\iimplies(y\iimplies x'))\wedge (y\iimplies (x\iimplies y')).
    \end{align*}
    Thus \(\varphi\) being a morphism implies that
    \begin{align}\label{bla}
        (x\iimplies x')\wedge (y\iimplies y')=(x\iimplies(y\iimplies x'))\wedge (y\iimplies (x\iimplies y'))
    \end{align}
    for all \(x\), \(x'\in X\) and \(y\), \(y'\in Y\). In particular, taking \(x=1=y'\) gives the first announced equality, while taking \(x'=1=y\) gives the second one.

    For the converse, the map \(\varphi\) has to be defined as \(\varphi(x,y)\coloneq x\wedge y\) for the same reasons as before. It obviously preserves the constant and the meet, and it makes the desired diagram commute. However, it preserves the implication if and only if the equation~\eqref{bla} holds for all \(x\), \(x' \in X\) and \(y\), \(y'\in Y\), which it does under the given assumptions.
\end{proof}

From this lemma, we can easily reprove the following fact we already knew from Section~\ref{section arithmetical}.

\begin{corollary}\label{corollary hslat no abelian objects}
    The category \(\HSLat\) of Heyting semilattices has no abelian objects besides the zero object.
\end{corollary}
\begin{proof}
    Let \(A\) be a Heyting semilattice such that \([A,A]=0\). Let \(a\in A\), then Lemma~\ref{lemma commute in hslat} implies that \(1=(a\iimplies a)=a\).
\end{proof}

We are also able to show that centralisers (Definition~\ref{Def Centraliser}) exist in the category of Heyting semilattices. Let \(X\oleq A\) be a subobject of a Heyting semilattice~\(A\). The centraliser of \(X\) in \(A\) is
\begin{align*}
    Z_A(X)=\{a\in A\mid \text{for each \(x\) in \(X\), \((x\iimplies a)=a\) and \((a\iimplies x)=x\)}\}\text{.}
\end{align*}
It is routine to check that it is closed under the meet in \(A\): for \(a\), \(a'\) in \(Z_A(X)\) and \(x\) in \(X\), use the identities \(x\iimplies (a\wedge a') = (x\iimplies a) \wedge (x\iimplies a')\) and \({(a\wedge a') \iimplies x} = {a\iimplies (a'\iimplies x)}\). To show it is closed under implications, for \(a\), \(a'\) and \(x\) as before, note that \(x\iimplies (a \iimplies a') = (x\iimplies a) \iimplies (x \iimplies a')=a\iimplies a'\), and \({(a\iimplies a') \iimplies x} = {(a\iimplies a') \iimplies (a \iimplies x)} = a\iimplies (a'\iimplies x) = a\iimplies x=x\). It now easily follows from Lemma~\ref{lemma commute in hslat} that is the largest subobject of \(A\), Huq-commuting with~\(X\).

\begin{proposition}\label{Proposition Centralisers}
    The category \(\HSLat\) is \defn{algebraically cartesian closed} \ACC. Moreover, centralisers of normal subobjects are normal in \(\HSLat\).
\end{proposition}
\begin{proof}
    The first part of the claim follows from Proposition 1.5 of~\cite{Bourn-Gray} which shows that in our context \ACC{} is equivalent to the existence of centralisers.

    To prove that \(Z_A(X)\) is normal (i.e., is a filter) as soon as \(X\) is, suppose that \(a\geq z\) for some \(a\in A\) and \(z\in Z_A(X)\). To see that \(a\in Z_A(X)\), let \(x\) be an element of \(X\). Since \(x\leq (x\iimplies a)\iimplies a\) it follows that \((x\iimplies a)\iimplies a\) is in \(X\). Therefore, we have that
    \begin{align*}
        ((x\iimplies a)\iimplies a) = (z \iimplies ((x\iimplies a)\iimplies a)) =  1
    \end{align*}
    where the last equality follows from the fact that \(z\leq a\leq (x\iimplies a)\iimplies a\). This means that \((x\iimplies a)\leq a \) and thus \((x\iimplies a)=a\) (because \((x \iimplies a)\geq a\) in general). On the other hand, \(z\leq a\) implies that
    \begin{align*}
        (a\iimplies x) \leq (z \iimplies x) =x
    \end{align*}
    and thus \((a\iimplies x) =x\) which concludes the proof.
\end{proof}

The following examples show that centralisers of non-normal monomorphisms need not be normal. Let us recall that in~\cite{MR0245486}, Nemitz called the term
\[
    x\psj y=((x\iimplies y)\iimplies y) \wedge ((y\iimplies x)\iimplies x)
\]
the \defn{pseudo-join} of \(x\) and \(y\) and showed that it has the following properties:
\begin{enumerate}
    \item \(x \leq x\psj y\) and \(y \leq x \psj y\);
    \item \(x\leq y\) if and only if \(x\psj y=y\);
    \item \(x\psj x=x\) and \(x\psj y= y\psj x\);
    \item \(x\psj (x\wedge y)=x=x\wedge(x\psj y)\).
\end{enumerate}
Let us also mention that it has the perhaps surprising property
\[
    x\iimplies (y\psj z) = (x\iimplies y)\psj (x\iimplies z).
\]

\begin{example}\label{Normal necessary for normal centraliser}
    If \(A\) is the Heyting semilattice with Hasse diagram
    
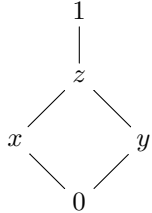
\[
        \xymatrix@!=0ex{
        & 1 &\\
        & z\ar@{-}[u] &\\
        x\ar@{-}[ur] & & y \ar@{-}[ul]\\
        & 0\ar@{-}[ul]\ar@{-}[ur] &
        }
    \]
    then \(X=\{x,1\}\) and \(Y=\{y,1\}\) Huq-commute in \(A\) because \(x\psj y=1\). Here we see that the pseudo-join is not the join. However, noting that \(x\psj z=z\) and \(x \psj 0=x\) it follows that \(Y\) is the centraliser of \(X\). However, \(Y\) is not up-closed in \(A\), so \(Z_A(X)\)  is not a normal subobject of \(X\).
\end{example}

Moreover, the pseudo-join can be used to capture Huq-commutativity:

\begin{proposition}\label{pseudo-join commutativity}
    A pair of sub-Heyting semilattices \(X\), \(Y\oleq A\) Huq-commute if and only if \(x\psj y=1\) for all \(x\in X\) and \(y\in  Y\).
\end{proposition}
\begin{proof}
    Just observe that \(x\iimplies y=y\) if and only if \((x\iimplies y)\iimplies y =1\), and \[{(x\iimplies y)\iimplies y=1=(y\iimplies x)\iimplies x}\]
    if and only if \(x\psj y=1\).
\end{proof}

We end with a general fact about commuting pairs of subobjects in arithmetical categories which applies in particular to Heyting semilattices.

\begin{proposition}\label{Prop Commute iff}
    Let \(\C\) be a semi-abelian arithmetical category. Two subobjects \(X\), \(Y\oleq A\) Huq-commute if and only if \(X \omeet Y =0\) and \(X\), \(Y\normal X \ojoin Y\).
\end{proposition}
\begin{proof}
    First we note that \(X\), \(Y\normal X \ojoin Y\) is equivalent to \(X^Y=X\) and \(Y^X=Y\). The subobjects \(X\) and \(Y\) of \(A\) Huq-commute if and only if the Higgins commutator \([X,Y]\oleq A\) vanishes.

    Now, using arithmeticity,
    \[
        X\omeet Y=[X\omeet Y,X\omeet Y]\oleq[X,Y]=0
    \]
    by Theorem~\ref{theorem arithm commu is intersection}. Hence by Proposition~\ref{Proposition Normal Closure}, together with the join decomposition formula and the rule for removing duplicates in a commutator (2.21 and 2.22 in~\cite{HVdL}) we find
    \begin{align*}
        X^Y & =[X,X\ojoin Y]=[X,X,Y]\ojoin [X,X]\ojoin [X,Y] \\
            & \oleq [X,Y]\ojoin [X,X]= [X,X]=X
    \end{align*}
    while of course \(X\oleq X^Y\). The equality \(Y^X=Y\) follows similarly.

    For the converse, we calculate \([X,Y]=X^Y\omeet Y^X=X\omeet Y=0\), again using arithmeticity.
\end{proof}

\begin{remark}\label{Remark Join in HSLat}
    At this point, it is useful to notice that in the category \(\HSLat\), the join \(X \ojoin Y\) of commuting subobjects \(X\), \(Y\oleq A\) is given by the set \(\{x \wedge y \mid \text{\(x \in X\) and \(y \in Y\)}\}\). Indeed, this set contains \(X\) and \(Y\), and is generated by them. Since it is closed under binary meets, it remains to verify that it is also closed under implications. To this end, suppose that \(x\), \(x' \in X\) and \(y\), \(y' \in Y\). Then
    \begin{align*}
        (x \wedge y) \iimplies (x' \wedge y') & = x \iimplies (y \iimplies (x' \wedge y'))               \\
                                              & = x \iimplies ((y \iimplies x') \wedge (y \iimplies y')) \\
                                              & = x \iimplies (x' \wedge (y \iimplies y'))               \\
                                              & = (x \iimplies x') \wedge (x\iimplies(y \iimplies y'))   \\
                                              & = (x \iimplies x') \wedge (y \iimplies y')
    \end{align*}
    where in the third and in the last equality we apply Lemma~\ref{lemma commute in hslat}.
\end{remark}

\section{\texorpdfstring{The conditions \SSH{} and \W{}}{The conditions (SSH) and (W)}}\label{section SSH and W}

It is known~\cite{MFVdL3} that all arithmetical semiabelian categories---including \(\HSLat\)---satisfy the \emph{Smith is Huq} condition \SH{} mentioned in the comment following Proposition~\ref{Smith implies Huq}. In the same paper, the authors introduce two new conditions denoted \SSH{} and~\W{} with the following implications between them:
\begin{align*}
    \W \To \SSH \To \SH.
\end{align*}
There it was claimed that the category \(\HSLat\) does not satisfy \SSH{}, and thus neither would it satisfy \W{}. However, the proof given there is inaccurate. (The function \(\gamma\) in that proof is easily seen not to preserve the operation ``\(\wedge\)'', so it is not a morphism.) The goal of this section is to provide a clarification on the matter. In Theorem~\ref{HSLat is SSH}, we prove that \(\HSLat\) does in fact satisfy \SSH{} (so not only the proof in~\cite{MFVdL3} is wrong, the result is wrong as well), while Corollary~\ref{HSLat is not W} confirms that it does not satisfy \W{}. This proves that the implication \(\W \To \SSH \) cannot be reversed, but leaves the question open for the implication \(\SSH \To \SH\).

\subsection*{\texorpdfstring{The condition \SSH{}}{The condition (SSH)}}
In this section we show that the category of Heyting semilattices satisfies \SSH{}.
Recall that a semi-abelian category satisfies \SSH{} if every change-of-base functor of fibration of points reflects Huq-commuting pairs of morphisms.

Examples of categories which satisfy this condition are groups, any category of interest in the sense of Orzech~\cite{Orzech}, any variety of non-associative algebras over a field~\cite[Proposition 7.7.3]{CocoThesis}, any abelian category, and Heyting semilattices---as we shall demonstrate now:

\begin{theorem}\label{HSLat is SSH}
    The category \(\HSLat\) of Heyting semilattices satisfies \SSH{}.
\end{theorem}

We have two proofs of this fact that are sufficiently different that we thought it interesting to include both.

\begin{proof}[Proof of Theorem~\ref{HSLat is SSH} via reflection of Huq-commuting pairs]
    We fix a cospan
    \[
        \xymatrix{
            X_1 \ar[r]^{\kappa_1}\ar[d]_{u_1} & A_1 \ar@<0.5ex>[r]^{\alpha_1}\ar[d]_{v_1} & B\ar@<0.5ex>[l]^{\beta_1}\ar@{=}[d]\\
            X \ar[r]^{\kappa} & A \ar@<0.5ex>[r]^{\alpha} & B\ar@<0.5ex>[l]^{\beta}\\
            X_2 \ar[r]_{\kappa_2}\ar[u]^{u_2} & A_2 \ar@<0.5ex>[r]^{\alpha_2}\ar[u]^{v_2} & B\ar@<0.5ex>[l]^{\beta_2}\ar@{=}[u]
        }
    \]
    of sub-split-extensions and suppose that \(u_1\) and \(u_2\) Huq-commute in \(X\). We must show that \(v_1\) and \(v_2\) Huq-commute in \((A,\alpha,\beta)\) as morphisms of points over \(B\). Without any loss of generality, we may suppose that all the morphisms in the previous diagram, excluding \(\alpha_1\), \(\alpha\) and \(\alpha_2\), are inclusions---here we use, for instance, that two morphisms Huq-commute if and only if their images do.

    By~\cite[Example 2.9.7]{Borceux-Bourn}, the category \(\Pt_B(\HSLat)\) is still arithmetical, so we can apply Proposition~\ref{Prop Commute iff} here to prove that \(v_1\) and \(v_2\) Huq-commute in it. Since kernel functors reflect zero objects (by protomodularity) and preserve intersections, \((A_1,\alpha_1,\beta_1)\omeet (A_2,\alpha_2,\beta_2)\) is trivial. Hence it remains to prove that \((A_1,\alpha_1,\beta_1)\) and \((A_2,\alpha_2,\beta_2)\) are normal in their join. Note here that joins in a category of points are computed in the base category. Since the category \(\HSLat\) is strongly protomodular by Theorem~\ref{theorem ToN consequences}, it is sufficient to show that the join of \((A_1,\alpha_1,\beta_1)\) and \((A_2,\alpha_2,\beta_2)\) is preserved by the kernel functor. In order to prove this, we will show that
    \[
        \widetilde{A} = \{((x_1\wedge x_2)\iimplies b) \wedge (x'_1 \wedge x'_2)\in A \mid \text{\(x_1\), \(x'_1 \in X_1\), \(x_2\), \(x'_2 \in X_2\) and \(b \in B\)}\}
    \]
    is a sub-Heyting-semilattice of \(A\), which together with the restrictions \(\widetilde{\alpha}\) and \(\widetilde{\beta}\) of \(\alpha\) and \(\beta\), respectively, forms the join of \(A_1\) and \(A_2\) in the category of points over \(B\).

    Firstly, note that \(\widetilde{A}\) contains \(A_1\) and \(A_2\), as can be seen from Remark~\ref{Remark Split Extension}. If it now forms a sub-Heyting-semilattice of \(X\), then it must be the join of \(A_1=X_1\ojoin B\) and \(A_2=X_2\ojoin B\), since it is generated by \(X_1\ojoin X_2\) and \(B\). Moreover, an element in \(\widetilde{A}\) is mapped to \(1\) under \(\alpha\) if and only if it is of the form \(x_1 \wedge x_2\) for some \(x_1 \in X_1\) and \(x_2 \in X_2\). Therefore, \(X \cap \widetilde{A}  = \Ker(\widetilde \alpha) = \{x_1 \wedge x_2 \mid \text{\(x_1 \in X_1\) and \(x_2 \in X_2\)}\}\) which, since \(u_1\) and \(u_2\) Huq-commute, is the join \(X_1 \ojoin X_2\) as explained in Remark~\ref{Remark Join in HSLat}. Thus, it only remains to be shown that \(\widetilde{A}\) is a sub-Heyting-semilattice. This is Lemma~\ref{SSH_sub-Heyting-semilattice} below.
\end{proof}

We deduce Lemma~\ref{SSH_sub-Heyting-semilattice} from a series of intermediate results.

\begin{lemma}\label{SSH_relative_comp_disjointish}
    For each \(x_1 \in X_1\), \(x_2\in X_2\) and \(b\in B\) we have
    \[
        (x_1\iimplies b) \wedge (x_2 \iimplies b) = b\text{.}
    \]
\end{lemma}
\begin{proof}
    Recalling that \(c \leq d\) if and only if \(c \iimplies d = 1\) and noting that
    \[
        b \leq (x_1 \iimplies b) \wedge (x_2 \iimplies b)
    \]
    we need only show that \(((x_1\iimplies b) \wedge (x_2\iimplies b)) \iimplies b =1\). Since \(x_2 \leqslant (x_2 \iimplies b) \iimplies b\), it follows that \((x_2 \iimplies b) \iimplies b \in X_2\), so that by Lemma~\ref{lemma commute in hslat} we have
    \[
        x_1 \iimplies ((x_2 \iimplies b) \iimplies b) = (x_2 \iimplies b) \iimplies b\text{.}
    \]
    Since \(a\mapsto (x_1 \iimplies a)\) is a morphism, we obtain
    \begin{align*}
        ((x_1\iimplies b) \wedge (x_2\iimplies b)) \iimplies  b & =(x_1\iimplies b) \iimplies ((x_2\iimplies b) \iimplies  b)                 \\
                                                                & =(x_1\iimplies b) \iimplies (x_1 \iimplies ((x_2\iimplies b) \iimplies  b)) \\
                                                                & = x_1 \iimplies (b \iimplies ((x_2\iimplies b) \iimplies b))                \\
                                                                & = 1
    \end{align*}
    as desired.
\end{proof}

\begin{lemma}\label{SSH_Lemma_A}
    For each \(x_1\in X_1\), \(x_2\in X_2\), and \(b\), \(b'\in B\) we have
    \[
        ((x_1\wedge x_2)\iimplies b) \iimplies b' = (x_2 \iimplies ((x_1\iimplies b)\iimplies b')) \wedge (x_1 \iimplies ((x_2\iimplies b)\iimplies b')).
    \]
    If, furthermore, \(b \leq b'\), then
    \begin{equation}\label{SSH_nice_equation}
        ((x_1\wedge x_2)\iimplies b) \iimplies b' = ((x_1\iimplies b)\iimplies b') \wedge ((x_2\iimplies b)\iimplies b')
    \end{equation}
    and hence
    \[
        ((x_1\wedge x_2)\iimplies b) \iimplies b' \in X_1 \ojoin X_2\text{.}
    \]
\end{lemma}
\begin{proof}
    By Lemma~\ref{SSH_relative_comp_disjointish} we have that
    \begin{align*}
         & ((x_1\wedge x_2)\iimplies b) \iimplies b'                                                                                                \\ & = ((x_1\wedge x_2)\iimplies b) \iimplies ((x_2 \iimplies b') \wedge (x_1 \iimplies b'))                                                  \\
         & =  ((x_2 \iimplies (x_1 \iimplies b)) \iimplies (x_2 \iimplies b')) \wedge ((x_1\iimplies(x_2\iimplies b)) \iimplies (x_1 \iimplies b')) \\
         & = (x_2 \iimplies ((x_1\iimplies b)\iimplies b')) \wedge (x_1 \iimplies ((x_2\iimplies b)\iimplies b')).
    \end{align*}
    The proof is completed by noting that if \(b\leq b'\), then since \(\alpha_1\)  applied to \({(x_1\iimplies b)\iimplies b'}\) is \(1\), it follows that \((x_1\iimplies b)\iimplies b'\) is in \(X_1\). Similarly,  \((x_2\iimplies b)\iimplies b'\) is in \(X_2\) and we obtain \eqref{SSH_nice_equation} by Lemma~\ref{lemma commute in hslat} applied to \(X_1\) and \(X_2\).
\end{proof}

\begin{lemma}\label{SSH_Lemma_B}
    For each \(x_1, x'_1 \in X_1\), \(x_2, x'_2 \in X_2\), and \(b, b' \in B\), if \(b \leq b'\), then
    \[
        ((x_1\wedge x_2)\iimplies b) \iimplies ((x'_1\wedge x'_2)\iimplies b') \in X_1\ojoin X_2.
    \]
\end{lemma}
\begin{proof}
    In any Heyting semilattice we have that \(x \iimplies (y \iimplies z) = y \iimplies (x \iimplies z)\), so
    \begin{align*}
        ((x_1\wedge x_2)\iimplies b) \iimplies ((x'_1\wedge x'_2)\iimplies b') & = (x'_1\wedge x'_2) \iimplies (((x_1\wedge x_2)\iimplies b)\iimplies b').
    \end{align*}
    The claim follows by Lemma~\ref{SSH_Lemma_A} and the fact that \(X_1\ojoin X_2\) is closed under the operation ``\(\iimplies\)''.
\end{proof}

\begin{lemma}\label{SSH_Lemma_C}
    For each \(x_1, x'_1 \in X_1\), \(x_2, x'_2 \in X_2\), and \(b\in B\),
    \[
        ((x_1\wedge x_2)\iimplies b) \iimplies (x'_1\wedge x'_2) \in X_1\ojoin X_2\text{.}
    \]
\end{lemma}
\begin{proof}
    Note that
    \[
        ((x_1\wedge x_2)\iimplies b) \iimplies (x'_1\wedge x'_2) = (((x_1\wedge x_2)\iimplies b) \iimplies x'_1 )\wedge (((x_1\wedge x_2)\iimplies b) \iimplies x'_2)\text{.}
    \]
    Since \(X_1\) and \(X_2\) Huq-commute in \(X\), Lemma~\ref{lemma commute in hslat} allows us to show that
    \begin{align*}
        ((x_1\wedge x_2)\iimplies b) \iimplies x'_2 & =  (x_1\iimplies (x_2\iimplies b)) \iimplies (x_1\iimplies x'_2) \\
                                                    & =  x_1\iimplies ((x_2\iimplies b) \iimplies x'_2)                \\
                                                    & =  (x_2\iimplies b) \iimplies x'_2\text{.}
    \end{align*}
    Note that \((x_2\iimplies b) \iimplies x'_2\) is an element of \(X_2\), since it is greater than \(x'_2\). This proves the last equality above, and at the same time tells us that \(((x_1\wedge x_2)\iimplies b) \iimplies x'_2\) belongs to \(X_2\). Likewise, \(((x_1\wedge x_2)\iimplies b) \iimplies x'_1 \) is in \(X_1\).
\end{proof}

\begin{lemma}\label{SSH_meets}
    \(\widetilde{A}\) is closed under meets.
\end{lemma}
\begin{proof}
    It is sufficient to prove that \(a\coloneq ((x_1\wedge x_2)\iimplies b) \wedge ((x'_1\wedge x'_2)\iimplies b')\) is in \(\widetilde{A}\) where
    \(x_1\), \(x'_1\) in \(X_1\), \(x_2\), \(x'_2\) in \(X_2\), and \(b\), \(b'\) in \(B\). Since the Johnstone terms \eqref{lemma protomodular terms for HSL} allow us to write
    \[
        a = ((a \iimplies \alpha(a))\iimplies \alpha(a))\wedge(((a \iimplies \alpha(a))\iimplies \alpha(a))\iimplies a)\text{,}
    \]
    we only need to show that both \(a \iimplies \alpha(a)\) and \(((a \iimplies \alpha(a))\iimplies \alpha(a))\iimplies a\) belong to \(X_1\ojoin X_2\). We have
    \begin{align*}
        a\iimplies \alpha(a) & = (((x_1\wedge x_2)\iimplies b) \wedge ((x'_1\wedge x'_2)\iimplies b'))\iimplies (b\wedge b')       \\
                             & = ((((x_1\wedge x_2)\iimplies b) \wedge ((x'_1\wedge x'_2)\iimplies b'))\iimplies b)                \\
                             & \qquad \wedge ((((x_1\wedge x_2)\iimplies b) \wedge ((x'_1\wedge x'_2)\iimplies b'))\iimplies b')   \\
                             & = (((x'_1\wedge x'_2)\iimplies b')\iimplies (((x_1\wedge x_2)\iimplies b)\iimplies b))              \\
                             & \qquad \wedge (((x_1\wedge x_2)\iimplies b)\iimplies (((x'_1\wedge x'_2)\iimplies b')\iimplies b'))
    \end{align*}
    which is in \(X_1\ojoin X_2\) by Lemma~\ref{SSH_Lemma_A} followed by Lemma~\ref{SSH_Lemma_C}. On the other hand,
    \begin{align*}
        ((a \iimplies \alpha(a)) \iimplies \alpha(a)) & \iimplies a                                                                                  \\
                                                      & = ((a \iimplies \alpha(a)) \iimplies (b \wedge b'))                                          \\
                                                      & \quad \iimplies (((x_1 \wedge x_2) \iimplies b) \,\wedge\, (x'_1 \wedge x'_2) \iimplies b')) \\
                                                      & = (((a \iimplies \alpha(a)) \iimplies (b \wedge b'))
        \iimplies ((x_1 \wedge x_2) \iimplies b))                                                                                                    \\
                                                      & \quad \wedge (((a \iimplies \alpha(a)) \iimplies (b \wedge b'))
        \iimplies ((x'_1 \wedge x'_2) \iimplies b')).
    \end{align*}
    This expression belongs to \(X_1 \ojoin X_2\), since \( a\iimplies \alpha(a) \in X_1 \ojoin X_2 \), by two applications of Lemma~\ref{SSH_Lemma_B}, because \(X_1 \ojoin X_2\) is closed under meets.
\end{proof}

Trivially, we have:
\begin{lemma}\label{SSH_Lemma_D}
    For each \(x\in X_1\ojoin X_2\) and \(a \in \widetilde{A}\) we have that \(x\iimplies a\) is in \(\widetilde{A}\).\noproof
\end{lemma}

\begin{lemma}\label{SSH_Lemma_E}
    For each \(x_1\in X_1\), \(x_2\in X_2\), and \(b,b'\in B\) we have
    \[
        ((x_1\wedge x_2)\iimplies b) \iimplies b' \in \widetilde{A}\text{.}
    \]
\end{lemma}
\begin{proof}
    By Lemma~\ref{SSH_Lemma_A} we have
    \[
        ((x_1\wedge x_2)\iimplies b) \iimplies b' = (x_2 \iimplies ((x_1\iimplies b)\iimplies b')) \wedge (x_1 \iimplies ((x_2\iimplies b)\iimplies b')).
    \]
    However, by Lemma~\ref{SSH_Lemma_D}, both \((x_2 \iimplies ((x_1\iimplies b)\iimplies b'))\) and \({(x_1 \iimplies ((x_2\iimplies b)\iimplies b'))}\) are in \(\widetilde{A}\), because \((x_i\iimplies b)\iimplies b'\in A_i\oleq \widetilde{A}\). The claim now follows from Lemma~\ref{SSH_meets}.
\end{proof}
\begin{lemma}\label{SSH_Lemma_F}
    For each \(x_1, x'_1\in X_1\), \(x_2, x'_2\in X_2\), and \(b, b'\in B\) we have
    \[
        ((x_1\wedge x_2)\iimplies b) \iimplies ((x'_1\wedge x'_2)\iimplies b') \in \widetilde{A}\text{.}
    \]
\end{lemma}
\begin{proof}
    We have that
    \begin{align*}
        ((x_1\wedge x_2)\iimplies b) \iimplies ((x'_1\wedge x'_2)\iimplies b') & =(x'_1\wedge x'_2) \iimplies (((x_1\wedge x_2)\iimplies b)\iimplies b')\text{,}
    \end{align*}
    which is in \(\widetilde{A}\) by Lemma~\ref{SSH_Lemma_E} followed by Lemma~\ref{SSH_Lemma_D}.
\end{proof}

\begin{lemma}\label{SSH_sub-Heyting-semilattice}
    \(\widetilde{A}\) is a sub-Heyting-semilattice of \(A\).
\end{lemma}
\begin{proof}
    Let \(x\), \(x'\), \(y\) and \(y'\) be elements of \(X_1\ojoin X_2\) and \(b\) and \(b'\) elements of \(B\). According to Lemma~\ref{SSH_meets}, we only need to show that \((x\wedge (y\iimplies b)) \iimplies (x'\wedge (y'\iimplies b'))\) is in \(\widetilde{A}\).
    We have
    \begin{align*}
         & (x\wedge (y\iimplies b)) \iimplies (x'\wedge (y'\iimplies b'))                                              \\
         & = ((x\wedge (y\iimplies b)) \iimplies x') \wedge ((x\wedge (y\iimplies b)) \iimplies (y'\iimplies b'))      \\
         & = (((y\iimplies b)\wedge x) \iimplies x') \wedge ((x\wedge (y\iimplies b)) \iimplies (y'\iimplies b'))      \\
         & = ((y\iimplies b) \iimplies (x\iimplies x'))\wedge (x\iimplies ((y\iimplies b)\iimplies (y'\iimplies b'))).
    \end{align*}
    The first term of this meet is in \(\widetilde{A}\) by Lemma~\ref{SSH_Lemma_C} and the second by Lemma~\ref{SSH_Lemma_F}
    combined with Lemma~\ref{SSH_Lemma_D}. We conclude by Lemma~\ref{SSH_meets}.
\end{proof}

\subsection*{A proof of \SSH{} via commutators}
The condition \SSH{} can be formulated in several equivalent ways. Our second proof that \(\HSLat\) satisfies it depends on the following one, which is a reformulation of item~\((iv)\) in Theorem~4.12 of~\cite{MFVdL3}:
\begin{quote}
    For any three subobjects \(X\), \(Y\), \(Z \oleq A\) of \(A\) such that \(X^Z = X\) and \(Y^Z = Y\), we have that \([X,Y] = 0\) implies~\([X,Y,Z] = 0\).
\end{quote}

We need the following lemma:

\begin{lemma}\label{lemmabottom}
    Let \(F_n\) be the free Heyting semilattice with \(n\) generators \(x_1, \dots , x_n\). Then \(F_n\) is a bounded Heyting semilattice with \(0= x_1 \wedge \dots \wedge x_n\).
\end{lemma}
\begin{proof}
    First, we recall that \(F_n\) is finite and thus bounded. Next, let \(x\in F_n\), then by \cite[Lemma~4.1]{Balb} we know that there exists a subset \(V\) of \(F_n\) such that
    \begin{align*}
        x= \bigwedge_{v\in V } \bigwedge_{i = 1,\dots, n}(v\iimplies x_i).
    \end{align*}
    However, for all \(i\) we have \( x_1 \wedge\dots\wedge x_n\leq v\iimplies x_i\) if and only if \( (x_1\wedge\dots\wedge x_n)\wedge v \leq x_i\) which is always true. Therefore, we proved that \(x_1\wedge\dots\wedge x_n\leq x\).
\end{proof}

\begin{proof}[Proof of Theorem~\ref{HSLat is SSH} via commutators]
    Let \(A\) be a Heyting semilattice and \(X\), \(Y\), \(Z\oleq A\) be three subobjects of it such that \(X^Z=X\), \(Y^Z=Y\) and \([X,Y]=0\). Given \(t\in [X,Y,Z]\), we want to prove that \(t=1\). The strategy of the proof is to first construct an element \(t'\) of the ternary commutator \([X,Y,Z]\) such that \(t'\leq t\). Then we show that under the assumptions \(X^Z=X\), \(Y^Z=Y\) and \([X,Y]=0\), we have that \(t'=1\), which allows us to conclude.

    \subsubsection*{Construction of \(t'\)}

    Since \(t\) is an element of the ternary commutator \([X,Y,Z]\), by definition it is the image of a word \(w={\Theta(\vec{x}, \vec{y}, \vec{z})}\in {X \diamond Y \diamond Z}\) where \(\vec{x}=(x_1,\dots, x_m)\in X^m\), \(\vec{y}=(y_1,\dots,y_n)\in Y^n\) and~\(\vec{z}=(z_1, \dots, z_p)\in Z^p\) for some natural numbers \(m\), \(n\), \(p\geq 1\).

    Let us fix some notation. For \(i=1,\dots , m\), we denote \(X_i\) to be the free algebra (in the variety \(\HSLat\)) on a single generator \(x_i\), i.e., \(X_i=\lbrace 1, x_i \rbrace\). In the same way we define  \(Y_j\) and \(Z_k\) for all \(j=1,\dots, n\) and \(k=1,\dots , p\). We write \(\widetilde{X}\coloneq \coprod_{i=1}^m X_i\), \(\widetilde{Y}\coloneq \coprod_{j=1}^n Y_j\), \(\widetilde{Z}\coloneq \coprod_{k=1}^p X_k\) and \(W\coloneq \widetilde{X} + \widetilde{Y} + \widetilde{Z}\). In \(W\), we consider the following three elements:
    \begin{align*}
        x\coloneq x_1 \wedge \dots \wedge x_m, \quad y\coloneq y_1\wedge\dots \wedge y_n, \quad z\coloneq z_1 \wedge \dots \wedge z_p.
    \end{align*}
    Now, since \(W\) is finitely generated (by the \(x_i\), \(y_j\) and \(z_k\)), we know from~\cite{McK, NeWh} that it is a finite distributive lattice, in which we can moreover~\cite[formula above Lemma~4.3]{Balb} compute the join \(w'\coloneq x\vee y\vee z\) with the following formula:
    \begin{align*}
        w'\coloneq x\vee y\vee z = \bigwedge_{s=x_1, \dots, z_p}\Bigl(\bigl(\bigwedge_{p=x,y,z} (p \iimplies s)\bigr)\iimplies s\Bigr).
    \end{align*}
    Using that \(x\iimplies x_i= y\iimplies y_j= z\iimplies z_k=1\) for all \(i\), \(j\), \(k\), we rewrite \(w'\) as
    \begin{align*}
        w'= & (\bigwedge_{i=1}^m ((y\iimplies x_i) \wedge (z\iimplies x_i))\iimplies x_i)        \\
            & \wedge (\bigwedge_{j=1}^n ((x\iimplies y_j) \wedge (z\iimplies y_j))\iimplies y_j) \\
            & \wedge (\bigwedge_{k=1}^p ((x\iimplies z_k) \wedge (y\iimplies z_k))\iimplies z_k)
    \end{align*}
    We claim that the image of the word \(w'\) in \(A\) is the needed element \(t'\).

    \subsubsection*{Is \(t'\) an element of \([X,Y,Z]\)?}

    It suffices to show that \(w'\) is an element of the ternary cosmash product \(\widetilde{X}\diamond \widetilde{Y} \diamond \widetilde{Z}\). First, it is indeed an element of \(\widetilde{X} + \widetilde{Y} + \widetilde{Z}\). Secondly, by replacing in \(w'\) each letter from \(\widetilde{X}\) by \(1\) we get:
    \begin{align*}
         & 1 \wedge (\bigwedge_{j=1}^n ((1\iimplies y_j) \wedge (z\iimplies y_j))\iimplies y_j) \wedge (\bigwedge_{k=1}^p ((1\iimplies z_k) \wedge (y\iimplies z_k))\iimplies z_k) \\
         & = \bigwedge_{j=1}^n (( y_j \wedge (z\iimplies y_j))\iimplies y_j) \wedge \bigwedge_{k=1}^p ((z_k \wedge (y\iimplies z_k))\iimplies z_k)                                 \\
         & = \bigwedge_{j=1}^n ( y_j \iimplies y_j) \wedge \bigwedge_{k=1}^p (z_k \iimplies z_k) =1.
    \end{align*}

    The same reasoning works to prove that replacing letters from \(Y\) or \(Z\) by \(1\) in the word \(w'\) makes it vanish. This shows that \(w'\in \widetilde{X}\diamond \widetilde{Y} \diamond \widetilde{Z}\) and thus \(t'\in [X,Y,Z]\).

    \subsubsection*{Is \(t'\leq t\) in \(A\)?}

    In order to prove this, we show that \(w'\leq w\) in \(W\). We conclude by the fact the canonical morphism from \(W\) to \(A\) preserves the order.

    First, we observe that both \(w\) and \(w'\) are elements of \([\widetilde{X},\widetilde{Y},\widetilde{Z}]\). By arithmeticity of the category \(\HSLat\), Theorem~\ref{theorem arithm n-comm} gives us that
    \begin{align*}
        [\widetilde{X},\widetilde{Y},\widetilde{Z}] = \widetilde{X}^{\widetilde{Y}\ojoin\widetilde{Z}} \omeet \widetilde{Y}^{\widetilde{X}\ojoin\widetilde{Z}} \omeet \widetilde{Z}^{\widetilde{X}\ojoin\widetilde{Y}}.
    \end{align*}
    Note that the smallest element of \(\widetilde{X}\), \(\widetilde{Y}\) and \(\widetilde{Z}\) is respectively \(x\), \(y\) and \(z\) (see Lemma~\ref{lemmabottom}). Therefore the smallest element of \(\widetilde{X}^{\widetilde{Y}\ojoin\widetilde{Z}}\), \(\widetilde{Y}^{\widetilde{X}\ojoin\widetilde{Z}}\) and \(\widetilde{Z}^{\widetilde{X}\ojoin\widetilde{Y}}\) is again \(x\), \(y\) and \(z\) respectively (see Lemma~\ref{lemma normal closure finite}). The smallest element of \([\widetilde{X},\widetilde{Y},\widetilde{Z}]\) should then be the smallest element greater than \(x\), \(y\) and \(z\). In other words, it is their join \(x\vee y \vee z=w'\). This implies that \(w'\leq w\) in \([\widetilde{X},\widetilde{Y},\widetilde{Z}]\) and thus in \(W\).

    \subsubsection*{Does \(X^Z=X\), \(Y^Z=Y\), \([X,Y]=0\) imply that \(t'=1\)?}

    First recall from Lemma~\ref{lemma commute in hslat} that \([X,Y]=0\) implies \(x\iimplies y =y\) and \(y\iimplies x=x\) for all \(x\in X\) and \(y\in Y\). Therefore, we have that for all \(i\)
    \begin{align*}
        ((y\iimplies x_i) \wedge (z\iimplies x_i))\iimplies x_i & = (x_i \wedge (z\iimplies x_i))\iimplies x_i \\
                                                                & = x_i \iimplies x_i =1.
    \end{align*}
    The same reasoning gives us
    \begin{align*}
        ((x\iimplies y_j) \wedge (z\iimplies y_j))\iimplies y_j =1
    \end{align*}
    for all \(j\). Under the condition \([X,Y]=0\), this allows us to rewrite \(t'\) as
    \begin{align*}
        t' = \bigwedge_{k=1}^p ((x\iimplies z_k) \wedge (y\iimplies z_k))\iimplies z_k.
    \end{align*}
    Next, using basic calculation rules of Heyting semilattices we get that
    \begin{align*}
        ((x\iimplies z_k) \wedge (y\iimplies z_k))\iimplies z_k & = ((x\iimplies z_k) \iimplies (y\iimplies z_k))\iimplies ((x\iimplies z_k)\iimplies z_k) \\
                                                                & = (y\iimplies ((x\iimplies z_k)\iimplies z_k))\iimplies ((x\iimplies z_k)\iimplies z_k)
    \end{align*}
    for all \(k\). We may see that, for all \(k\), \((x\iimplies z_k)\iimplies z_k\) is in \(X^Z\), since it is the image of a word in \(Z\flat X\): it suffices to replace \(x\) by \(1\). Using the hypothesis~\(X^Z=X\), we can rewrite \((x\iimplies z_k)\iimplies z_k\) as \(x_k'\in X\) for all \(k\). Thus, under our assumptions, \(t'\) rewrites as
    \begin{align*}
        t'= \bigwedge_{k=1}^p (y\iimplies x'_k)\iimplies x'_k.
    \end{align*}
    This last expression of \(t'\) vanishes, since \((y\iimplies x'_k)\iimplies x'_k \in [X,Y]=0\) for all \(k\).
\end{proof}

In~\cite{MFVdL3}, the category \(\HSLat\) was used in an attempt to prove that the converse of the implication \(\SSH \To \SH\) does not hold. Knowing now that this category does in fact satisfy \SSH{}, it remains an open problem to find either a counterexample which separates these conditions or a proof of their equivalence.

\subsection*{The condition (W)}
In this section, we provide several reformulations of the condition \W{}, introduced and studied in~\cite{MFVdL3}, and we show that it does not hold in the category of Heyting semilattices.

Let us first recall two reformulations—one for \SH{} and one for \W{}—which are in the same spirit as the one used for \SSH{} in the proof of Theorem~\ref{HSLat is SSH} via commutators. Thanks to item~\((vii)\) in Proposition~3.5 of~\cite{MFVdL3}, the condition \SH{} can be reformulated as:
\begin{quote}
    For any three subobjects \(X\), \(Y\), \(Z \oleq A\) such that \(X^A = X\) and \(Y^A = Y\), we have that \([X,Y] = 0\) implies \([X,Y,Z] = 0\).
\end{quote}
Similarly, thanks to item~\((iii)\) in Proposition~5.1 of~\cite{MFVdL3}, the condition \W{} can be reformulated as:
\begin{quote}
    For any three subobjects \(X\), \(Y\), \(Z \oleq A\), we have that \([X,Y] = 0\) implies \([X,Y,Z] = 0\).
\end{quote}
These reformulations make the implications
\begin{align*}
    \W \To \SSH \To \SH
\end{align*}
particularly clear.

\begin{theorem}
    Let \(\C\) be a semi-abelian category. The following conditions are equivalent:
    \begin{tfae}
        \item \(\C\) satisfies \W{};
        \item \(\C\) satisfies \SH{}, and for all subobjects \(X, Y \oleq A\), we have
        \[
            [X,Y] = 0 \quad \Leftrightarrow \quad [X^A, Y^A] = 0.
        \]
    \end{tfae}
\end{theorem}

\begin{proof}
    (i) \(\Rightarrow\) (ii): It is clear that \W{} implies \SH{}. Moreover, by monotonicity of the Higgins commutator, we immediately have that \([X^A, Y^A] = 0\) implies \([ X, Y ] = 0\).

    Now, suppose \( [X,Y] = 0 \). We want to show that \( [X^A, Y^A] = 0 \). Applying Corollary~\ref{Corollary MM-NC} together with the join decomposition formula (Proposition~2.22 in~\cite{HVdL}), we obtain:
    \begin{align*}
        [X^A,Y^A]
         & = [X \ojoin [X,A],\; Y \ojoin [Y,A]]                                    \\
         & = [X,[X,A],Y] \ojoin [X,[X,A],[Y,A]] \ojoin [X,[X,A],Y,[Y,A]]           \\
         & \quad \ojoin [X,Y] \ojoin [X,Y,[Y,A]] \ojoin [X,[Y,A]] \ojoin [[X,A],Y] \\
         & \quad \ojoin [[X,A],[Y,A]] \ojoin [[X,A],Y,[Y,A]].
    \end{align*}

    Since \([X,Y] = 0\), we may apply items~(iv) and~(v) of Proposition~2.21 in~\cite{HVdL} to deduce that each term in the decomposition above is contained in the ternary commutator \([X,Y,A]\). But under the assumption that \(\C\) satisfies \W{}, this commutator vanishes. Hence \([X^A, Y^A] = 0\), as required.

    (ii) {\(\To\)} (i): Let \(X, Y, Z \oleq A\) be subobjects such that \([X,Y] = 0\). By hypothesis, we have \([X^A, Y^A] = 0\), and since \(\C\) satisfies \SH{}, it follows that \([X^A, Y^A, Z] = 0\). But the ternary commutator \([X,Y,Z]\) is contained in \([X^A, Y^A, Z]\), so we conclude that \([X,Y,Z] = 0\), as required.
\end{proof}

\begin{corollary}\label{HSLat is not W}
    The category \(\HSLat\) does not satisfy \W{}.
\end{corollary}
\begin{proof}
    In Example~\ref{Normal necessary for normal centraliser} the subobjects \(X\) and \(Y\) commute, but their normal closures do not.
\end{proof}

\begin{corollary}\label{Cor Cosmash Assoc}
    The category \(\HSLat\) of Heyting semilattices is not \emph{cosmash associative} (see \cite{Cosmash} and the comment right above Definition~\ref{Def Higher Higgins}).
\end{corollary}
\begin{proof}
    It is shown in~\cite{CocoThesis} that cosmash associativity implies \W{}.
\end{proof}

\section{Algebraic Coherence}\label{section AC}

Introduced in~\cite{acc}, in the context of a semi-abelian category the condition called \defn{algebraic coherence} \AC{} has strong categorical-algebraic consequences such as \SH{}, \SSH{} and \NH{}. Nevertheless, there are many examples of such categories: all varieties of groups, Lie algebras, Poisson algebras, rings without unit, associative algebras; the category of cocommutative Hopf algebras over a field of characteristic zero; all abelian categories; etc. Non-examples include loops and digroups (because \SH{} is not satisfied), non-associative rings, \(n\)-Lie algebras and Jordan algebras. In~\cite{acc}, it is mentioned (Examples~4.10) that the category \(\HSLat\) is not \AC{} because it does not satisfy \SSH{} as claimed in~\cite{MFVdL3}. However, as we already explained in Section~\ref{section SSH and W}, this claim is false and thus a clarification is needed. The goal of the present section is to confirm that Heyting semilattices do not form an algebraically coherent category. \emph{A fortiori}, \(\HSLat\) is not an \emph{algebraic logos} in the sense of~\cite{AlgebraicLogoi}; neither is it \emph{locally algebraically cartesian closed} \LACC{} in the sense of~\cite{GrayPhD,Gray2012,Bourn-Gray}.

Let us recall that a category \(\C\) with finite limits is said to be \defn{algebraically coherent} if the change-of-base functors of its fibration of points preserve jointly strongly epimorphic pairs of arrows.

\begin{theorem}\label{HSLat not AC}
    The variety \(\HSLat\) of Heyting semilattices is not algebraically coherent.
\end{theorem}
\begin{proof}
    We prove that for \(B = \{b,1\}\), the functor \(\Ker_B \colon \Pt_B(\HSLat) \rightarrow \HSLat\) does not preserve binary joins of subobjects, which implies that jointly strongly epimorphic pairs of arrows are not preserved.

    Let \(A\) be the Heyting semilattice
    
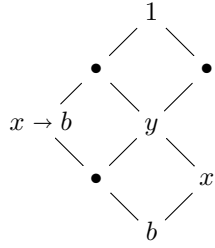
\[
        \xymatrix@!@C=-2ex@R=-2ex{
        & & 1 \\
        & \bullet\ar@{-}[ur] & & \bullet\ar@{-}[ul] \\
        x\iimplies b \ar@{-}[ur] && y \ar@{-}[ul]\ar@{-}[ur] && \\
        & \bullet\ar@{-}[ul]\ar@{-}[ur] & & x \ar@{-}[ul] \\
        & & b \ar@{-}[ul]\ar@{-}[ur]
        }
    \]
    Let \(\beta\) be the inclusion of \(B\) into \(A\), and let \(\alpha \colon A \to B\) be the canonical quotient of the filter \(X = \ua x\) onto \(B\) (see Section~\ref{Section Normal Subobjects} for why this is indeed the quotient). A~straightforward computation shows that the subset \(\{b, y, x\}\) generates \(A\).

    Define \(A_1 \coloneqq \{x, b, x\iimplies b, 1\}\) and \(A_2 \coloneqq \{y, b, 1\}\), two sub-Heyting semilattices of~\(A\). It then follows that \(A = A_1 \ojoin A_2\).

    Let \(i_1 \colon A_1 \to A\), \(i_2 \colon A_2 \to A\), \(\beta_1 \colon B \to A_1\), and \(\beta_2 \colon B \to A_2\) be the respective inclusion maps. Then \((A_1, \alpha i_1, \beta_1) \ojoin (A_2, \alpha i_2, \beta_2) = (A, \alpha, \beta)\) in \(\Pt_B(\HSLat)\), since joins in a category of points are computed in the base category.

    Now observe that \(X_1 = \Ker(\alpha i_1) = \{x,1\}\) and \(X_2 = \Ker(\alpha i_2) = \{y,1\}\). However, \(X_1 \ojoin X_2 = \{x, y, 1\} \neq X\), which concludes the proof.
\end{proof}

We conclude that the implication \(\AC \To \SSH\) proved in~\cite{acc} cannot be reversed; in other words, \SSH{} is strictly weaker than \AC.

\section{Normality of unions and action accessibility}\label{Section Normalisers and AA}

In~\cite{BJK2}, the authors define a semi-abelian category to satisfy \defn{normality of unions} when for each commutative square of the form
\[
    \xymatrix{
    X\ar@{{ >}->}[d]_{m} \ar@{{ >}->}[r]^{l} & B \ar@{{ >}->}[d]^{s_B}
    \\
    C \ar@{{ >}->}[r]_-{s_C} & A
    }
\]
where all arrows are monomorphisms, if \(X\) is a normal subobject of both \(B\) and~\(C\), then \(X\) is a normal subobject in the union (join) \(B\ojoin C\) taken in \(A\). In the article, the authors explain the relationship between this condition and what they call \emph{representability of actions}.

We give a counterexample showing that this condition is not satisfied by the semi-abelian category \(\HSLat\).

\begin{theorem}\label{theorem normality of unions}
    The category \(\HSLat\) of Heyting semilattices does not satisfy normality of unions.
\end{theorem}
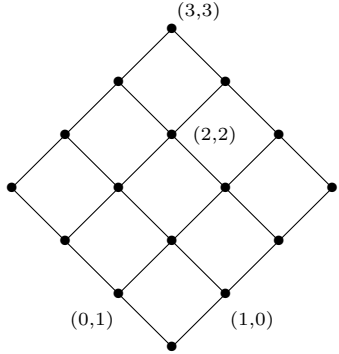
\begin{figure}
    \[
        \xymatrix@!0@=2em{&&& *=0{\bullet} \ar@{-}[ld] \ar@{}[r]^{(3,3)} \ar@{-}[rd] &\\
        &&*=0{\bullet} \ar@{-}[ld] \ar@{-}[rd] &&*=0{\bullet} \ar@{-}[ld] \ar@{-}[rd] \\
        &*=0{\bullet} \ar@{-}[ld] \ar@{-}[rd] &&*=0{\bullet} \ar@{-}[ld] \ar@{-}[rd] \ar@{}[rr]|(.4){({2},{2})} &&*=0{\bullet} \ar@{-}[ld] \ar@{-}[rd]\\
        *=0{\bullet}\ar@{-}[rd]&&*=0{\bullet} \ar@{-}[ld] \ar@{-}[rd] &&*=0{\bullet} \ar@{-}[ld] \ar@{-}[rd] &&*=0{\bullet} \ar@{-}[ld] \\
        &*=0{\bullet} \ar@{-}[rd] &&*=0{\bullet} \ar@{-}[ld] \ar@{-}[rd] &&*=0{\bullet} \ar@{-}[ld]\\
        &&*=0{\bullet} \ar@{}[ld]|{(0,{1})} \ar@{-}[rd]&&*=0{\bullet} \ar@{-}[ld] \ar@{}[rd]|{({1},0)} \\
        &&&*=0{\bullet}&&}
    \]
    \caption{The Heyting semilattice \(A= C_4 \times C_4\) }\label{Figure A}
\end{figure}
\begin{proof}
    Let \(A\) be the Heyting semilattice in Figure~\ref{Figure A}, defined as \(A\coloneq C_4 \times C_4\) where \(C_4=\{ 0, 1, 2, 3\}\) is the chain of length \(4\). We consider \(X\), \(B\), and \(C\), three subobjects of \(A\) defined as follows:
    \begin{align*}
        X & \coloneq \{ ({2}, {2}), (3,3)\},         \\
        B & \coloneq \{ (0,{1}), ({2},{2}), (3,3)\}, \\
        C & \coloneq \{ ({1},0), ({2},{2}),(3,3) \}.
    \end{align*}
    It is clear that \(X\) is up-closed in both \(B\) and \(C\), i.e., \(X\normal B\) and \(X\normal C\). A short computation shows that \(B\ojoin C=A\). We conclude by observing that \(X\) is not up-closed in \(A\), so \(X\) is not normal in \(B\ojoin C\).
\end{proof}

\begin{remark}\label{remark normaliser}
    \emph{Normalisers} in semi-abelian categories were introduced in~\cite{MR3275285} in order to generalise the classical notion of the category of groups. Using Proposition~2.10 from~\cite{MR3275285}, one can see that with the previous theorem we have also proved that the normaliser of \(X\) in \(A\) does not exist. Therefore, another way of expressing Theorem~\ref{theorem normality of unions} is by saying that in \(\HSLat\), normalisers do not exist in general.
\end{remark}
Let us recall that a semi-abelian category \(\C\) is said to be \defn{action accessible} if for each \(X\in \C\) the category \(\SpltExt_\C(X)\) of split extensions with kernel \(X\) has enough sub-terminal objects. This notion was first introduced in~\cite{BJ07} in order to, amongst other things, allow the calculation of centralisers of equivalence relations and of subobjects. There, the authors also proved that action accessibility implies \SH. We refer to it for further details.

\begin{theorem}\label{theorem action accessibility}
    The category \(\HSLat\) of Heyting semilattices is not action accessible.
\end{theorem}
\begin{proof}
    We consider the same objects and subobjects as in the proof of Theorem~\ref{theorem normality of unions}. As we observed in Remark~\ref{remark normaliser}, the normaliser of \(X\) in \(A\) does not exist. However, noting that \(X\) is the image of a normal subobject along the diagonal, we conclude via \cite[Theorem~3.1]{GRAY:NCAAA} that the category of Heyting semilattices is not action accessible.
\end{proof}

In particular, this means that the known implication between action accessibility and the condition requiring the existence of centralisers of normal monomorphisms which are themselves normal---see~\cite{MR2745553} and~\cite{MR3275285}---is strict.

Note that \(\HSLat\), being an arithmetical category, trivially has \emph{representable representations} in the sense that the actions on an abelian object are representable~\cite{BJK2,Edinburgh}. On the other hand, it is impossible for all actions to be representable, because this would imply action accessibility~\cite{BJ07}. So \(\HSLat\) is not \emph{action representable}. This shows that a category may have representable representations without being an action accessible/representable category.

\section{Open questions}\label{section open questions}
\subsection{}\label{First question}
Theorem~\ref{HSLat is SSH} tells us that the category \(\HSLat\) satisfies the condition \SSH. The question arises, whether the arithmeticity of \(\HSLat\) explains this fact. In other words, do all arithmetical semi-abelian categories satisfy \SSH? We have no proof, but no counterexample either; all arithmetical semi-abelian categories we know of satisfy the condition, and it is a known fact~\cite{MFVdL3} that arithmeticity implies the condition~\SH. Same question for \ToN.

\subsection{}
The example in~\cite{MFVdL3} which we refer to in Section~\ref{section SSH and W} was meant to separate \SSH{} from \SH, showing that \SSH{} is strictly stronger. Now that we know this example is not valid, the question is open: either find a proof that \SH{} implies \SSH, or a counterexample separating the two conditions.

\subsection{}
Essentially the same question as in~\ref{First question} may be asked for strong protomodularity: Do all Moore categories in the sense of~\cite{Rodelo:Moore} satisfy \SSH? Closely related general questions are: Does \W{} imply \AC{}?
Does \AC{} imply \(\NH^n\)?

\subsection{}
It is not clear to us how, for some given \(n\), the conditions \(\NH^n\) and \(\NH^{n-1}\) which occur at the end of Section~\ref{section arithmetical} are connected.

\subsection{}
Finally, it would be helpful to have a more precise description of the commutators in \(\HSLat\)---one which subsumes Lemma~\ref{lemma commute in hslat}, but gives a complete description of the Higgins or Huq commutator object. If we consider the subobjects \(X=\{x, 1\}\) and \(Y=\{ y, 1\}\) of \(F_2\) (the free Heyting semilattice on \(\{x,y\}\)), then \([ X,Y ]= \ua{w}\) where \(w\coloneq ((x \iimplies y) \iimplies y) \wedge ((y \iimplies x) \iimplies x)\). It can, moreover, be seen that this commutator \([X,Y]\) is generated by the elements \((x \iimplies y) \iimplies y\) and \((y \iimplies x) \iimplies x\). Those two elements vanish if and only if \(x\iimplies y=y\) and \({y\iimplies x=x}\). Therefore, making the parallel with the case of groups---where commutators are generated by elements of the form \(xyx^{-1}y^{-1}\) and \(yxy^{-1}x^{-1}\)---it seems to us that commutators in Heyting semilattices should be generated by elements of the form \((x \iimplies y) \iimplies y\) and \((y \iimplies x) \iimplies x\) in general. Unfortunately, we were not able to prove this.

\section*{Acknowledgements}
We are grateful to Nelson Martins-Ferreira for fruitful discussions on several aspects of the paper.

%\bibliography{tim}

\begin{thebibliography}{10}

    \bibitem{Balb}
    R.~Balbes, \emph{On free pseudo-complemented and relatively pseudo-complemented semi-lattices}, Fund. Math. \textbf{78} (1973), no.~2, 119--131.

    \bibitem{BlokPigozzi}
    W.~J. Blok and D.~Pigozzi, \emph{On the structure of varieties with equationally definable principal congruences {III}}, Algebra Universalis \textbf{32} (1994), 545--608.

    \bibitem{Borceux-Bourn}
    F.~Borceux and D.~Bourn, \emph{Mal'cev, protomodular, homological and semi-abelian categories}, Math. Appl., vol. 566, Kluwer Acad. Publ., 2004.

    \bibitem{BJK2}
    F.~Borceux, G.~Janelidze, and G.~M. Kelly, \emph{On the representability of actions in a semi-abelian category}, Theory Appl. Categ. \textbf{14} (2005), no.~11, 244--286.

    \bibitem{Bourn2001b}
    D.~Bourn, \emph{A categorical genealogy for the congruence distributive property}, Theory Appl. Categ. \textbf{8} (2001), no.~14, 391--407.

    \bibitem{B4}
    D.~Bourn, \emph{Commutator theory in strongly protomodular categories}, Theory Appl. Categ. \textbf{13} (2004), no.~2, 27--40.

    \bibitem{Bourn:Direct-Image}
    D.~Bourn, \emph{On the direct image of intersections in exact homological categories}, J.~Pure Appl.\ Algebra \textbf{196} (2005), 39--52.

    \bibitem{MR2745553}
    D.~Bourn, \emph{Centralizer and faithful groupoid}, J. Algebra \textbf{328} (2011), 43--76.

    \bibitem{AlgebraicLogoi}
    D.~Bourn, A.~S. Cigoli, J.~R.~A. Gray, and T.~Van~der Linden, \emph{Algebraic logoi}, J. Pure Appl. Algebra \textbf{227} (2023), 107293.

    \bibitem{BG}
    D.~Bourn and M.~Gran, \emph{Centrality and normality in protomodular categories}, Theory Appl. Categ. \textbf{9} (2002), no.~8, 151--165.

    \bibitem{Bourn-Gray}
    D.~Bourn and J.~R.~A. Gray, \emph{Aspects of algebraic exponentiation}, Bull. Belg. Math. Soc. Simon Stevin \textbf{19} (2012), 823--846.

    \bibitem{Bourn-Janelidze}
    D.~Bourn and G.~Janelidze, \emph{Characterization of protomodular varieties of universal algebras}, Theory Appl. Categ. \textbf{11} (2003), no.~6, 143--147.

    \bibitem{BJ07}
    D.~Bourn and G.~Janelidze, \emph{Centralizers in action accessible categories}, Cah. Topol. G{\'e}om. Differ. Cat{\'e}g. \textbf{L} (2009), no.~3, 211--232.

    \bibitem{Smash}
    A.~Carboni and G.~Janelidze, \emph{Smash product of pointed objects in lextensive categories}, J.~Pure Appl.~Algebra \textbf{183} (2003), 27--43.

    \bibitem{AlanThesis}
    A.~S. Cigoli, \emph{Centrality via internal actions and action accessibility via centralizers}, Ph.D. thesis, Universit\`a degli Studi di Milano, 2009.

    \bibitem{acc}
    A.~S. Cigoli, J.~R.~A. Gray, and T.~Van~der Linden, \emph{Algebraically coherent categories}, Theory Appl. Categ. \textbf{30} (2015), no.~54, 1864--1905.

    \bibitem{CGrayVdL1}
    A.~S. Cigoli, J.~R.~A. Gray, and T.~Van~der Linden, \emph{On the normality of {H}iggins commutators}, J.~Pure Appl.\ Algebra \textbf{219} (2015), 897--912.

    \bibitem{CigoliMontoliCharSubobjects}
    A.~S. Cigoli and A.~Montoli, \emph{Characteristic subobjects in semi-abelian categories}, Theory Appl.~Categ. \textbf{30} (2015), no.~7, 206--228.

    \bibitem{MMClementinoAMontoliLSousa2015}
    M.~M. Clementino, A.~Montoli, and L.~Sousa, \emph{Semidirect products of (topological) semi-abelian algebras}, J. Pure Appl. Algebra \textbf{219} (2015), no.~1, 183--197.

    \bibitem{EverVdLRCT}
    T.~Everaert and T.~Van~der Linden, \emph{Relative commutator theory in semi-abelian categories}, J.~Pure Appl. Algebra \textbf{216} (2012), no.~8--9, 1791--1806.

    \bibitem{Edinburgh}
    X.~Garc{\'\i}a-Mart{\'\i}nez, M.~Tsishyn, T.~Van~der Linden, and C.~Vienne, \emph{Algebras with representable representations}, Proc.\ Edinb.\ Math.\ Soc.\ (2) \textbf{64} (2021), 555--573.

    \bibitem{GM-VdL2}
    X.~Garc{\'\i}a-Mart{\'\i}nez and T.~Van~der Linden, \emph{A characterisation of {L}ie algebras amongst anti-commutative algebras}, {J.~Pure} Appl. Algebra \textbf{223} (2019), 4857--4870.

    \bibitem{Grandis-HA2}
    M.~Grandis, \emph{Homological algebra, in strongly non-abelian settings}, World Scientific Publishing, Singapore, 2013.

    \bibitem{GrayPhD}
    J.~R.~A. Gray, \emph{Algebraic exponentiation in general categories}, Ph.D. thesis, University of Cape Town, 2010.

    \bibitem{Gray2012}
    J.~R.~A. Gray, \emph{Algebraic exponentiation in general categories}, Appl.\ Categ.\ Structures \textbf{20} (2012), 543--567.

    \bibitem{MR3275285}
    J.~R.~A. Gray, \emph{Normalizers, centralizers and action representability in semi-abelian categories}, Appl. Categ. Structures \textbf{22} (2014), no.~5-6, 981--1007.

    \bibitem{GRAY:NCAAA}
    J.~R.~A. Gray, \emph{Normalizers, centralizers and action accessibility}, Theory Appl.\ Categ. \textbf{30} (2015), no.~12, 410--432.

    \bibitem{GrayVdL1}
    J.~R.~A. Gray and T.~Van~der Linden, \emph{Peri-abelian categories and the universal central extension condition}, J. Pure Appl. Algebra \textbf{219} (2015), no.~7, 2506--2520.

    \bibitem{HVdL}
    M.~Hartl and T.~Van~der Linden, \emph{The ternary commutator obstruction for internal crossed modules}, Adv.~Math. \textbf{232} (2013), no.~1, 571--607.

    \bibitem{Higgins}
    P.~J. Higgins, \emph{Groups with multiple operators}, Proc. Lond. Math. Soc. (3) \textbf{6} (1956), no.~3, 366--416.

    \bibitem{Huq}
    S.~A. Huq, \emph{Commutator, nilpotency and solvability in categories}, Q. J. Math. \textbf{19} (1968), no.~2, 363--389.

    \bibitem{Janelidze}
    G.~Janelidze, \emph{Internal crossed modules}, Georgian Math. J. \textbf{10} (2003), no.~1, 99--114.

    \bibitem{Janelidze-Marki-Tholen}
    G.~Janelidze, L.~M{\'a}rki, and W.~Tholen, \emph{Semi-abelian categories}, J.~Pure Appl. Algebra \textbf{168} (2002), no.~2--3, 367--386.

    \bibitem{JQT}
    R.~Jansen, M.~Qasim, and W.~Tholen, \emph{The normal decomposition of a morphism in categories without zeros}, J. Algebra \textbf{678} (2025), 392--425.

    \bibitem{Johnstone:Heyting}
    P.~T. Johnstone, \emph{A note on the semiabelian variety of {H}eyting semilattices}, {G}alois Theory, {H}opf Algebras, and Semiabelian Categories, Fields Inst. Commun., vol.~43, Amer. Math. Soc., 2004, pp.~317--318.

    \bibitem{Kohler}
    P.~K\"{o}hler, \emph{Brouwerian semilattices}, Trans. Amer. Math. Soc. \textbf{268} (1981), no.~1, 103--126.

    \bibitem{LMS}
    S.~Lapenta, G.~Metere, and L.~Spada, \emph{Relative ideals in homological categories, with an application to {MV}-algebras}, Theory Appl.~Categ. \textbf{41} (2024), no.~27, 878--893.

    \bibitem{Maltsev-Sbornik}
    A.~I. Mal'cev, \emph{On the general theory of algebraic systems}, Mat. Sbornik N. S. \textbf{35} (1954), no.~6, 3--20.

    \bibitem{MM-NC}
    S.~Mantovani and G.~Metere, \emph{Normalities and commutators}, J.~Algebra \textbf{324} (2010), no.~9, 2568--2588.

    \bibitem{MFVdL}
    N.~Martins-Ferreira and T.~Van~der Linden, \emph{A note on the ``{S}mith is {H}uq'' condition}, Appl.\ Categ.\ Structures \textbf{20} (2012), no.~2, 175--187.

    \bibitem{MFVdL3}
    N.~Martins-Ferreira and T.~Van~der Linden, \emph{Further remarks on the ``{S}mith is {H}uq'' condition}, Appl.\ Categ.\ Structures \textbf{23} (2015), no.~4, 527--541.

    \bibitem{McK}
    C.~G. McKay, \emph{The decidability of certain intermediate propositional logics}, J. Symbolic Logic \textbf{33} (1968), 258--264.

    \bibitem{NeWh}
    W.~Nemitz and T.~Whaley, \emph{Varieties of implicative semi-lattices. {II}}, Pacific J. Math. \textbf{45} (1973), 303--311.

    \bibitem{Nemitz}
    W.~C. Nemitz, \emph{Implicative semi-lattices}, Trans.\ Amer.\ Math.\ Soc. \textbf{117} (1965), 128--142.

    \bibitem{MR0245486}
    W.~C. Nemitz, \emph{Semi-{B}oolean lattices}, Notre Dame J. Formal Logic \textbf{10} (1969), 235--238.

    \bibitem{Orzech}
    G.~Orzech, \emph{Obstruction theory in algebraic categories {I} and {II}}, J.~Pure Appl. Algebra \textbf{2} (1972), 287--314 and 315--340.

    \bibitem{Pedicchio}
    M.~C. Pedicchio, \emph{A categorical approach to commutator theory}, J.~Algebra \textbf{177} (1995), 647--657.

    \bibitem{Pedicchio2}
    M.~C. Pedicchio, \emph{Arithmetical categories and commutator theory}, Appl. Categ. Structures \textbf{4} (1996), no.~2--3, 297--305.

    \bibitem{PVdL2}
    G.~Peschke and T.~Van~der Linden, \emph{A homological view of categorical algebra}, preprint \texttt{arXiv:2404.15896}, 2024.

    \bibitem{Pixley}
    A.~F. Pixley, \emph{Distributivity and permutability of congruence relations in equational classes of algebras}, Proc. Amer. Math. Soc. \textbf{14} (1963), 105--109.

    \bibitem{Cosmash}
    Ü. Reimaa, T.~Van~der Linden, and C.~Vienne, \emph{Associativity and the cosmash product in operadic varieties of algebras}, Illinois J. Math. \textbf{67} (2023), no.~3, 563--598.

    \bibitem{Rodelo:Moore}
    D.~Rodelo, \emph{Moore categories}, Theory Appl. Categ. \textbf{12} (2004), no.~6, 237--247.

    \bibitem{RVdL3}
    D.~Rodelo and T.~Van~der Linden, \emph{Higher central extensions via commutators}, Theory Appl. Categ. \textbf{27} (2012), no.~9, 189--209.

    \bibitem{RVdL2}
    D.~Rodelo and T.~Van~der Linden, \emph{Higher central extensions and cohomology}, Adv. Math. \textbf{287} (2016), 31--108.

    \bibitem{SVdL3}
    C.~S. Simeu and T.~Van~der Linden, \emph{On the ``{T}hree {S}ubobjects {L}emma'' and its higher-order generalisations}, J.\ Algebra \textbf{546} (2020), 315--340.

    \bibitem{Smith}
    J.~D.~H. Smith, \emph{{M}al'cev varieties}, Lecture Notes in Math., vol. 554, Springer, 1976.

    \bibitem{CocoThesis}
    C.~Vienne, \emph{Categorical-algebraic conditions in semi-abelian categories}, Ph.D. thesis, UClouvain, 2025.

\end{thebibliography}
%\bibliographystyle{amsplain-nodash}
%% .bbl:

\end{document}